\documentclass[a4paper,oneside,english,reqno]{amsart}

\usepackage[utf8]{inputenx}
\usepackage{amsthm, amsmath, amssymb, amsfonts, mathrsfs, mathtools, stmaryrd, yhmath}
\usepackage{babel, textcomp, url, enumerate, latexsym, graphicx, varioref, hyperref, multirow, layout, siunitx, booktabs}
\usepackage{pdflscape}
\usepackage{verbatim}
\usepackage{geometry}
\usepackage{dsfont}
\usepackage{csquotes}
\usepackage{xcolor}

\RequirePackage{enumitem}
\setlist[itemize]{font = \upshape, before = \leavevmode}
\setlist[enumerate]{font = \upshape, before = \leavevmode}
\setlist[description]{before = \leavevmode}

\usepackage{tikz, babel, rotating, tikz-cd, pigpen}
\usepackage{tikz}
\usetikzlibrary{matrix, arrows, decorations.pathmorphing}
\usepackage{cleveref}
\RequirePackage[color       = blue!20,
                bordercolor = black,
                textsize    = tiny,
                figwidth    = 0.99\linewidth]{todonotes}
\usepackage[all]{xy}
\usepackage{microtype}
\usepackage{csquotes}
\usepackage{bbm}
\RequirePackage[backend    = biber,
                sortcites  = true,
                giveninits = true,
                doi        = false,
                isbn       = false,
                url        = false,
                maxnames   = 6,
                style      = numeric]{biblatex}
\numberwithin{equation}{section}

\theoremstyle{plain}
\newtheorem*{theorem*}{Theorem}
\newtheorem*{corollary*}{Corollary}
\newtheorem{theorem}{Theorem}[section]
\newtheorem{lemma}[theorem]{Lemma}
\crefname{lemma}{Lemma}{Lemmas}
\newtheorem{sblemma}[theorem]{Sublemma}
\crefname{sblemma}{Sublemma}{Sublemmas}

\newtheorem{proposition}[theorem]{Proposition}
\newtheorem*{proposition*}{Proposition}
\crefname{proposition}{Proposition}{Propositions}

\newtheorem{corollary}[theorem]{Corollary}

\theoremstyle{definition}
\newtheorem{definition}[theorem]{Definition}

\newtheorem{example}[theorem]{Example}

\theoremstyle{remark}
\newtheorem{remark}[theorem]{Remark}

\newcommand{\Spec}{\operatorname{Spec}}
\newcommand{\colim}{\operatorname{colim}}

\newcommand{\Sch}{\mathrm{Sch}}
\newcommand{\Sm}{\mathrm{Sm}}
\newcommand{\SmAff}{\mathrm{SmAff}}
\newcommand{\EssSm}{\mathrm{EssSm}}
\newcommand{\Et}{\mathrm{Et}}
\newcommand{\locsh}{\mathrm{l.sh}}
\newcommand{\Smlh}{\Sm^\locsh}

\newcommand{\A}{\mathbb A}
\newcommand{\PP}{\mathbb P}

\newcommand{\Gm}{{\mathbb G_m}}
\newcommand{\Gmup}{\mathbb G_m}
\newcommand{\Gmwo}{{\Gmup^{\wedge 1}}}
\newcommand{\pt}{\mathrm{pt}}

\newcommand{\et}{{\widehat{\text{\'et}}}}
\newcommand{\cet}{\text{\'et}}
\newcommand{\hypet}{\et}
\newcommand{\nis}{\mathrm{Nis}}
\newcommand{\Nis}{\mathrm{Nis}}
\newcommand{\Zar}{\mathrm{Zar}}

\newcommand{\fr}{\mathrm{fr}}
\newcommand{\Fr}{\mathrm{Fr}}
\newcommand{\ZF}{\bbZ\mathrm{F}}
\newcommand{\ovZF}{\overline{\bbZ\mathrm{F}}}

\newcommand{\Func}{\mathrm{Func}}

\newcommand{\bbC}{\mathbb C}

\newcommand{\bbZ}{\mathbb Z}
\newcommand{\Ab}{\mathrm{Ab}}

\newcommand{\Corr}{\mathrm{Corr}}

\newcommand{\SH}{\mathrm{SH}}
\newcommand{\DM}{\mathbf{DM}}

\newcommand{\calSP}{{\mathcal S}}
\newcommand{\calC}{{\mathcal C}}

\newcommand{\Spc}{\mathrm{Spc}}
\newcommand{\Spcpointed}{\Spc^\bullet}

\newcommand{\PShAet}{\Spc_{\Ao,\et}}
\newcommand{\Spcfr}{\Spc^\fr}

\newcommand{\SpcAet}{\PShAet}

\newcommand{\gp}{\mathrm{gp}}
\newcommand{\gpfr}{\mathrm{gpfr}}
\newcommand{\frgp}{{\fr,\gp}}
\newcommand{\Spcfrgp}{\mathrm{Spc}^{\gpfr}}

\newcommand{\PShfrgp}{\Spcfrgp}
\newcommand{\PShfrgpA}{\Spcfrgp_{\Ao}}
\newcommand{\PShfrgpAS}{\Spcfrgp_{\Ao,\Sigma}}
\newcommand{\PShfrgpS}{\Spcfrgp_{\Sigma}}

\newcommand{\PShfrgpAnis}{\Spcfrgp_{\Ao,\nis}}

\newcommand{\SpcfrgphatnAet}{\Spcfrgp_{\Ao,\et,\compln}}

\newcommand{\SpcGfrgphatnAet}{\Spc^{\bbG_m^{-1},\gpfr}_{\Ao,\et,\compln}}

\newcommand{\SptGfrhatnAet}{\Spt^{\bbG_m^{-1},\fr}_{\Ao,\et,\compln}}

\newcommand{\Sptfr}{\SptSfr}

\newcommand{\Spt}{\mathrm{Spt}}
\newcommand{\SptS}{\Spt}
\newcommand{\Sptgeq}{\Spt_\geq}

\newcommand{\bdleqgeq}{\mathrm{b}}
\newcommand{\Sptb}{\Spt_\bdleqgeq}

\newcommand{\SptSet}{\SptS_{\et}}

\newcommand{\SptShatnhypet}{\Spt_{\et,\compln}}

\newcommand{\SptSetleq}{\SptS_{\et\leq}}
\newcommand{\SptSAnis}{\SptS_{\Ao,\nis}}
\newcommand{\SptSAet}{\SptS_{\Ao,\et}}
\newcommand{\SptSetA}{\SptS_{\Ao,\et}}
\newcommand{\SptHZ}{\Spt^{\bbZ}}

\newcommand{\SptHZn}{\Spt^{\bbZ/l}} 
\newcommand{\SptSfr}{\SptS^\fr}
\newcommand{\SptSfrSigma}{\SptS^\fr_{\Sigma}}
\newcommand{\SptShatnfrSigma}{\SptS^\fr_{\Sigma,\compln}}

\newcommand{\SptShatnfrAS}{\mathrm{Spt}^\fr_{\Ao,\Sigma,\compln}}

\newcommand{\SptSfrAet}{\mathrm{Spt}^\fr_{\Ao,\et}}

\newcommand{\SptHZfr}{\Spt^{\bbZ\fr}}
\newcommand{\SptHZnfr}{\Spt^{\bbZ\fr/l}}

\newcommand{\HZ}{\mathrm{H}\bbZ}
\newcommand{\HZn}{\mathrm{H}\bbZn}

\newcommand{\SptShereetA}{\Spt^{\Sphere^{-1}}_{\Aet}}

\newcommand{\SptSphereAet}{\Spt^{\Sphere}_{\Ao,\et}}
\newcommand{\SptSpherefrAet}{\Spt^{\Sphere,\fr}_{\Ao,\et}}
\newcommand{\SptSpherefrgpAet}{\Spt^{\Sphere,\gpfr}_{\Ao,\et}}

\newcommand{\SptPhatnAhypet}{\Spt^{\PP^{-1}}_{\Ao,\et,\compln}}
\newcommand{\SpcPhatnAhypet}{\Spc^{\PP^{-1}}_{\Ao,\et,\compln}}
\newcommand{\SptSGhatnAhypet}{\Spt^{\bbG_m^{-1}}_{\Ao,\et,\compln}}
\newcommand{\SptSpherehatnAhypet}{\Spt^{\Sphere}_{\Ao,\et,\compln}}
\newcommand{\SptShatnAnis}{\Spt_{\Ao,\nis,\compln}}
\newcommand{\SptfrShatnAhypet}{\Spt^{\fr}_{\Ao,\et,\compln}}
\newcommand{\SptShatnAhypet}{\Spt_{\Ao,\et,\compln}}

\newcommand{\Aet}{\Ao,\et}
\newcommand{\Anis}{\Ao,\nis}

\newcommand{\SptPhatnfrgpAet}{\Spt^{\PP^{-1},\gpfr}_{\Aet,\compln}}
\newcommand{\SptSpherehatnfrgpAet}{\Spt^{\Sphere,\gpfr}_{\Aet,\compln}}
\newcommand{\SptPhatnfrAet}{\Spt^{\PP^{-1},\fr}_{\Aet,\compln}}
\newcommand{\SptPfrgphatnAet}{\SptPhatnfrgpAet}
\newcommand{\SptSpherefrgphatnAet}{\SptSpherehatnfrgpAet}

\newcommand{\SptPfrgpAnis}{\Spt^{\PP^{-1},\gpfr}_{\Ao,\Nis}}

\newcommand{\SpcPAet}{\Spc^{\PP^{-1}}_{\Ao,\et}}
\newcommand{\SpcPetA}{\Spc^{\PP^{-1}}_{\Ao,\et}}

\newcommand{\SptGAet}{\Spt^{\bbG_m^{-1}}_{\Ao,\et}}
\newcommand{\SptSGAet}{\Spt^{\bbG_m^{-1}}_{\Ao,\et}}
\newcommand{\SptPANis}{\Spt^{\PP^{-1}}_{\Ao,\Nis}}
\newcommand{\SptPAnis}{\Spt^{\PP^{-1}}_{\Ao,\Nis}}
\newcommand{\SpcPANis}{\Spc^{\PP^{-1}}_{\Ao,\Nis}}

\newcommand{\SptSGAnis}{\Spt^{\bbG_m^{-1}}_{\Ao,\Nis}}

\newcommand{\SptPfrAet}{\Spt^{\PP^{-1},\fr}_{\Ao,\et}}
\newcommand{\SptGfrAet}{\Spt^{\bbG_m^{-1},\fr}_{\Ao,\et}}

\newcommand{\ZPre}{\mathrm{Ab}}

\newcommand{\ZnPSh}{\ZPre^{\bbZ/l}}
\newcommand{\ZPShfr}{\ZPre^{\fr}}

\newcommand{\ZPShfrASigma}{\ZPShfr_{\Ao,\Sigma}}\newcommand{\ZnPShfrASigma}{\ZnPShfr_{\Ao,\Sigma}}

\newcommand{\ZnPShfr}{\ZPre^{\fr/l}}

\newcommand{\ZnPShASigma}{\ZPre^{\bbZ/l}_{\Ao,\Sigma}}

\newcommand{\lcet}{{\mathrm{l.c.}\et}}

\DeclareMathOperator{\cofib}{cofib}
\newcommand{\Id}{\mathrm{Id}}

\newcommand{\Ker}{\operatorname{Ker}}
\newcommand{\Coker}{\operatorname{Coker}}

\newcommand{\Map}{\mathrm{Map}}

\newcommand{\Zgpfr}{{\bbZ\gpfr}}
\newcommand{\HZgpfr}{{\bbZ\gpfr}}

\newcommand{\hZgpfr}{h^{\bbZ\gpfr}}

\newcommand{\Sigmaconst}{\Xi}
\newcommand{\Pconst}{\Sigmaconst}
\newcommand{\Betti}{\mathrm{Be}}

\newcommand{\Cor}{\mathrm{Cor}}

\newcommand{\echark}{\operatorname{echar}}

\newcommand{\CMon}{\mathrm{CMon}}

\newcommand{\bbG}{\mathbb{G}}

\newcommand{\Gal}{\mathrm{Gal}}

\newcommand{\compln}{{\hat{l}}}

\newcommand{\taugeq}[2]{{#1}{#2}}
\newcommand{\tauleq}[2]{{#1}{#2}}
\newcommand{\etgeq}{\taugeq{\et}{\geq}}
\newcommand{\etleq}{\tauleq{\et}{\leq}}

\newcommand{\Nisgeq}{\taugeq{\Nis}{\geq}}
\newcommand{\Nisleq}{\tauleq{\Nis}{\leq}}

\newcommand{\Sigmaleq}{\tauleq{\Sigma}{\leq}}

\newcommand{\Nisgeql}[1]{\taugeq{\Nis}{\geq #1}}

\newcommand{\ptauleql}[1]{\taugeq{\tau}{\leq #1}}
\newcommand{\ptauleq}{\taugeq{\tau}{\leq}}
\newcommand{\ptaugeql}[1]{\taugeq{\tau}{\geq#1}}
\newcommand{\ptaugeq}{\taugeq{\tau}{\geq}}

\newcommand{\etbdleqgeq}{\taugeq{\et}{\bdleqgeq}}
\newcommand{\ptaubdleqgeq}{\taugeq{\tau}{\bdleqgeq}}

\newcommand{\calL}{\mathcal{L}}
\newcommand{\Lrep}{\calL}
\newcommand{\calS}{\mathcal{S}}

\newcommand{\Str}{S}
\newcommand{\Fog}{F}

\newcommand{\infSigmaSfrgp}{\Sigma^{\infty,\gpfr}_{S^1}}
\newcommand{\infOmegaSfrgp}{\Omega^{\infty,\gpfr}_{S^1}}

\renewcommand{\Re}{Re}

\newcommand{\RigidityEquivalence}{rigidity equivalence } 
\newcommand{\RigidityTheorem}{rigidity theorem } 
\newcommand{\ComparisonTheorems}{comparison theorems } 
\newcommand{\SHITheorem}{strict homotopy invariance theorem} 
\newcommand{\Cancellation}{cancellation} 
\newcommand{\CancellationTheorem}{cancellation theorem}  

\newcommand{\ReconstructionEquivalence}{reconstruction equivalence }

\newcommand{\leftadjright}{\rightleftarrows}
\newcommand{\rightadjleft}{\leftrightarrows}

\newcommand{\essemb}{\mathrm{ess}}

\newcommand{\eff}{\mathbf{eff}} 
\newcommand{\veff}{\mathbf{veff}} 

\newcommand{\SptPAniseff}{\Spt^{\PP^{-1}}_{\eff,\sAo,\nis}} 

\newcommand{\SptSGAniseff}{\Spt^{\bbG_m^{-1}}_{\eff,\sAo,\nis}}
\newcommand{\SptSGhatnAniseff}{\Spt^{\bbG_m^{-1}}_{\eff,\sAo,\nis,\compln}}
\newcommand{\SptSGhatnAeteff}{\Spt^{\bbG_m^{-1}}_{\eff,\sAo,\et,\compln}}

\newcommand{\SptPAnisveff}{\Spt^{\PP^{-1}}_{\veff,\sAo,\nis}} 
\newcommand{\SptPAetveff}{\Spt^{\PP^{-1}}_{\veff,\sAo,\et}} 

\newcommand{\SpcPhatnAnis}{\Spc^{\PP^{-1}}_{\Ao,\nis,\compln}} 
\newcommand{\SpcPAniseff}{\Spc^{\PP^{-1}}_{\eff,\sAo,\nis}} 
\newcommand{\SptPhatnAetveff}{\Spt^{\PP^{-1}}_{\veff,\sAo,\et,\compln}} 

\newcommand{\Ao}{\A^1} 
\newcommand{\sAo}{\,\A^1} 

\newcommand{\hetl}{h. \'et. l. }

\newcommand{\Sphere}{\mathbb S}

\newcommand{\motpref}{\Gm,\A^1}
\newcommand{\Pmotpref}{\PP^1,\A^1}
\newcommand{\unsmotpref}{\gp,\A^1}

\newcommand{\bbZn}{\bbZ/l}

\author{Andrei Druzhinin}
\address{Chebyshev Laboratory, St. Petersburg State University  \& 
St. Petersburg Department of Steklov Mathematical Institute of Russian Academy of Sciences, Russia}
\email{\href{mailto:andrei.druzh@gmail.com}{andrei.druzh@gmail.com}}

\author{Ola Sande}
\address{University of Oslo, Oslo, Norway}
\email{\href{mailto:olas1802@gmail.com}{olas1802@gmail.com}}

\begin{document}


\title[Hyper\'et. framed motives and comparison of st. hom. groups]{Hypercomplete \'etale framed motives and comparison of stable homotopy groups of motivic spectra and \'etale realizations over a field.
} 



\begin{abstract}
For any base field and integer $l$ invertible in $k$, we prove that $\Omega^\infty_{\mathbb{G}_m}$ and
$\Omega^\infty_{\PP^1}$ commute with hyper \'etale sheafification $L_{\et}$ and Betti realization through infinite loop space theory in motivic homotopy theory.

The central subject of this article is an $l$-complete hypercomplete \'etale analog of the framed motives theory developed by Garkusha and Panin. Using Bachman's hypercomplete \'etale \RigidityTheorem and
the $\infty$-categorical approach of framed motivic spaces by Elmanto, Hoyois, Khan, Sosnilo, Yakerson, we prove the recognition principle and
the framed motives formula for the composite functor
    \[
    \Delta^\mathrm{op}\Sm_k\to \SptSGAet(\Sm_k)\xrightarrow{\Omega^\infty_{\Gm}} \SptShatnhypet(\Sm_k).
    \]

The first applications include the hypercomplete \'etale stable motivic connectivity theorem and an \'etale local isomorphism
\[
\pi^{\Ao,\Nis}_{i,j}(E)\simeq\pi^{\Ao,\et}_{i,j}(E)
\]
for any $l$-complete effective motivic spectra $E$, and $j\geq 0$.
Furthermore, we obtain a new proof for Levine's comparison isomorphism over $\mathbb C$, $\pi_{i,0}^{\Ao,\nis}(E)(\bbC)\cong \pi_i(\Betti(E))$,
and Zargar's generalization for algebraically closed fields, that applies to an arbitrary base field.
\end{abstract}

\maketitle

\tableofcontents

\maketitle

\tableofcontents


\section{Introduction}

\'Etale and $l$-adic cohomology theories 
were originally introduced to play the role of an analog of
singular cohomology of topological spaces for algebraic varieties
over a field $k$, 
see \cite{MilneLectEtCoh}\cite{EberhardFreitag}\cite{deligneetale}.
Motivic homotopy theory 
explains
this algebro geometric analogy
on an advanced level,
because
motivic
constructions and results
look
even more similar 
and
fragmentary 
almost 
parallel to 
the 
topological one's. Also,
$l$-complete stable motivic invariants 
for an integer $l$ invertible in $k$
are 
more universal 
than $l$-adic ones,
since
$l$-complete 
locally constant \'etale hypersheaves of spectra are motivically local,
and moreover, 
there is an 
embedding
\begin{equation}\label{eq:embeddinlcetShvtoSmot}
\SptS_{\lcet,\hat{l}}(\Sm_k)\simeq
\SptShatnhypet(\Et_k)\xrightarrow{\Pconst^{\et,\compln}_{\Anis}} 
\SptSGAnis(\Sm_k)=\Func_{\Anis}(\Sm_k^\mathrm{op},\Spt)[\bbG_m^{\wedge -1}],
\end{equation}
see Notation 
\ref{notation:Pconst} and \eqref{eq:SptPANis}, \eqref{eq:intro:Rigidity:SptPAetSmkSptSAetSmkSptSetEtk}
below for details.
A fruitful study is the comparison 
between these three
areas, classical homotopy theory, motivic homotopy theory and $l$-adic theory, which has enabled 
 computation of motivic 
invariants
by reduction to the topological or \'etale local ones.
The basic input for 
such a comparison result is 
an ``\'etale realization functor''
which is 
left adjoint to 
the composite in
\eqref{eq:embeddinlcetShvtoSmot}, it
preserves the symmetric monoidal structure,
and 
takes motivic suspension to the simplicial suspension.
In this paper, we delve deeper into the study of \'etale realization functors and harness several comparison results for arbitrary fields. Our main technical machinery is hyperétale framed motives theory, see \Cref{sect:FrMotConnect}, \Cref{cor:etalefraedmotivesformula:Fr}.


For $k=\bbC$, the \'etale realization is the Betti realization,
induced by
inverse images
along the functor
$\Sm_{\bbC}\to\Spc; X\mapsto X(\bbC)$,
because for any separably closed $k$,
$\Spc(\Et_k)\simeq\Spc.$
%
Betti realization functors play a crucial role in understanding various motivic homotopy categories. For instance, the works of Levine \cite{Levine_2013} and Heller-Ormsby \cite{zbMATH06629254} demonstrate this by computing stable motivic homotopy groups of the complex and real numbers, respectively, via 
comparison results for Betti realizations. 
These functors also feature prominently in the closely related area of logarithmic motivic homotopy theory \cite{zbMATH07533305}. 

The mentioned above
Levine's comparison theorem \cite{Levine_2013} 
for stable motivic homotopy groups
uses
the comparison theorem 
for singular homologies of algebraic varieties 
by Suslin and Voevodsky
\cite{SusVoev}
written
much earlier
and
prior to
the constructions of motivic categories 
$\DM(k)$, $\mathbf{H}(k)$, $\mathbf{SH}(k)$
\cite{Voe-motives,Morel-Voevodsky,Jardine-spt,morel-trieste,mot-functors,Cisinski-Deglise-Triangmixedmotives}.
Implicitly in \cite{SusVoev} was 
a part of future construction of the \'etale Voevodsky motives category without $\Gm$-stabilization.
The 
constructions of the categories $\DM^-(k)$ and $\DM_{\cet}^-(k,\bbZ/l)$
were done in
\cite{Voe-motives,MVW}
around the millennium
with 
the comparison result for the \'etale realization functor is defined as the composite
\begin{equation}\label{eq:DMReetRidg}\DM^-(k)\to\DM^-_{\cet}(k,\bbZ/l)\simeq \mathbf{D}^-(\Gal(k),\bbZ/l),\end{equation}
where the right side equivalence is called Suslin-Veovodsky rigidity theorem.
In 2014, Ayoub 
\cite{etalerealizationAyoub}
obtained 
a rigidity theorem and 
the \'etale realization functor like \eqref{eq:DMReetRidg}
for $\mathbf{DA}^-(S,\bbZ/l)$ over a scheme $S$.
%
The result for $\DM_{\cet}(S,\bbZ/l)$ over base schemes $S$
is proven 
by Cisinski and D\'eglise in \cite{zbMATH06578150} published in 2016.

In 2004,
Isaksen’s construction \cite{zbMATH02081955} 
based on shape theory 
provided 
``\'etale realization functors''
landing in pro-spaces 
for the unstable Morel-Voevodsky
motivic homotopy category 
$\mathbf{H}(S)$ over a scheme $S$.
This construction
along with
its Hoyois's $\infty$-categorical incarnation \cite{HoyoisIsaksenConstruction} 
and 
the
motivic stabilization \cite[\S 4]{zbMATH07103864},
was used by Zargar \cite{zbMATH07103864} in 2019
to generalize Levine's results \cite{Levine_2013}
to an arbitrary algebraically closed field $k$.

Rigidity theorems by Bachmann \cite{BachmannRigidity,bachmann2021remarks} 
generalizing 
\cite{Voe-motives,MVW,Cisinski-Deglise-Triangmixedmotives,etalerealizationAyoub}
provide
hyper\'etale realization functors 
$\Re^{\PP^1,\A^1,\Nis}_{\et,\hat{l}}$
and
$\Re^{S^1,\A^1,\Nis}_{\et,\hat{l}}$
for
stable motivic homotopy categories
$\mathbf{SH}_{\Anis}(S)$ and $\mathbf{SH}^{S^1}_{\Anis}(S)$ over a scheme $S$, see \eqref{eq:intro:def_RNishypetcompln}.
%
\Cref{th:piAniswnBetenhypetANDLoopspectraandspacesANDRealisationfunctors}(a)
proves
the comparison theorem for stable motivic homotopy groups of smooth schemes
over
any field $k$
and
$l\in\bbZ$ invertible in $k$.
For
separably closed fields $k$,
this 
implies the result 
with respect to 
\'etale realization functors
by
Isaksen mentioned above,
i.e. the isomorphism with stable homotopy groups of $l$-complete \'etale homotopy type,
generalizing Zargar's result \cite{zbMATH07103864}.
The 
proof
approach
is different 
from
the ones 
by 
Levine 
\cite{Levine_2013}
and 
Zargar  
\cite{zbMATH07103864},
and
the proof 
is formally
independent from these results, and moreover, the results
for
the category 
$\DM(k)$ 
in
\cite{SusVoev,MVW}.
See \Cref{sect:ComparisonoftheapproachesforComparisonTheorems} for comparison of the strategies,
and see also \Cref{rem:basechangefieldextensionArgumentcompariosn} 
for suggestion of
an alternative 
way of the proof, which would use \cite{zbMATH07103864}.

In \Cref{th:Remonoidadjoint},
we extend the rigidity theorem and construction of \'etale realization,
see \eqref{eq:intro:def_RNishypetcompln_CMongp},
to 
the motivic homotopy category of
grouplike
Segal's $\Gamma$-spaces
over a field $k$
proving
a hypercomplete \'etale counterpart to the stable connectivity theorem from \cite{connectivityMorel},
see
\Cref{th:connectivity}(a) of \Cref{cor:etalefraedmotivesformula:Fr}%
.
This allows to define the arrow

\begin{equation}\label{eq:ReOmegatoOmegaRe}
\Re^{\gp,\A^1,\Nis}_{\et,\hat{l}}\Omega^\infty_{\PP^1}
\to
\Omega_{S^1}^\infty\Re^{\PP^1,\A^1,\Nis}_{\et,\hat{l}}.
\end{equation}
\Cref{th:piAniswnBetenhypetANDLoopspectraandspacesANDRealisationfunctors}(b) 
proves 
that \eqref{eq:ReOmegatoOmegaRe} is an equivalence.

By the way
we prove a comparison result
for
hyper\'etale and 
Nisnevich 
local
$l$-complete stable motivic homotopy categories,
see \Cref{cor:nisconget},
using 
framed motives theory
provided for these categories
by
\cite{garkusha2018framed} and
\Cref{cor:etalefraedmotivesformula:Fr},
respectively,
see
\Cref{sect:FrMotConnect}.

\subsection{Construction of $l$-complete hypercomplete \'etale realization functors}\label{sect:intro:Renisetcompln}
The hypercomplete \'etale stable motivic homotopy category over a base field $k$ fits into the middle of the diagram
\begin{equation}\label{eq:SptPANis}
\xymatrixcolsep{5pc}\xymatrix{
\SptSGAnis(\Sm_k)\ar@<0.5ex>[r]^{L^{\motpref,\nis}_{\et}}&\SptSGAet(\Sm_k)\ar@<0.5ex>[l]^{E_{\motpref,\nis}^{\et}}\ar@<-0.5ex>[r]_{\Gamma^{\motpref,\et}_{\et}}&\Spt_\et(\Et_k)\ar@<-0.5ex>[l]_{\Xi^{\et}_{\motpref,\et}}
}
\end{equation}
between
the $\infty$-category of hypercomplete \'etale spectra on the small etale site $\Spt_\et(\Et_k)$
and the
stable motivic homotopy category
with respect to the Nisnevich topology,
whose
$\infty$-categorical incarnations \cite{robalo2013noncommutative}, 
\cite[\S 4.1]{bachmann2020norms} 
over a field $k$
we denote by
\[\begin{array}{lcll}
\Spc_{\Anis}(k)&=&\Func_{\Anis}(\Sm_k^\mathrm{op},\Spc)&,\\
\Spt_{\Anis}(k)&=&\Func_{\Anis}(\Sm_k^\mathrm{op},\Spt)&,\\
\SpcPANis(\Sm_k)&=&\Func_{\Anis}(\Sm_k^\mathrm{op},\Spc)[\PP^{\wedge -1}]&,\\
\SptSGAnis(\Sm_k)&=&\Func_{\Anis}(\Sm_k^\mathrm{op},\Spt)[\bbG_m^{\wedge -1}]&,\\
\SptPANis(\Sm_k)&=&\Func_{\Anis}(\Sm_k^\mathrm{op},\Spt)[\PP^{\wedge -1}]
\end{array}\]
for motivic spaces,
$S^1$-spectra,
$\PP^{\wedge 1}$-spectra,
$(S^1,\bbG_m^{\wedge 1})$-bispectra,
$(S^1,\PP^{\wedge 1})$-bispectra,
respectively.

The adjunction to the right \eqref{eq:SptPANis}, see Notation \ref{notation:Pconst} and \ref{notation:functorsGammas},
is an equivalence up to $l$-completion
for $l\in\bbZ$ invertible in $k$
by the \RigidityTheorem \cite{BachmannRigidity,bachmann2021remarks} by Bachmann.
Moreover, by \cite[Theorem 3.1]{bachmann2021remarks}
there are equivalences
\begin{equation}\label{eq:intro:Rigidity:SptPAetSmkSptSAetSmkSptSetEtk} 
    \SptSGhatnAhypet(\Sm_k)
    \stackrel{\simeq}{\leftadjright}
    \SptShatnAhypet(\Sm_k)
    \stackrel{\simeq}{\rightadjleft}
    \SptShatnhypet(\Et_k),
    \end{equation}
where the subscript $\compln$ denotes the $l$-completion.
This leads to the construction of \'etale realization functor as the composite
\begin{equation}\label{eq:intro:def_RNishypetcompln}\Re^{\motpref,\Nis}_{\hypet,\compln}\colon\SptSGAnis(\Sm_k)\xrightarrow{L^{\motpref,\nis}_{\et,\compln}} \SptSGhatnAhypet(\Sm_k)\xrightarrow{(\Pconst^{\et,\compln}_{\motpref,\et,\compln})^{-1}} \SptShatnhypet(\Et_k),\end{equation} 
see \cite[Theorem 7.1]{BachmannRigidity} or \cite[Def. 2.4.6]{haine2023spectral},
and similarly for $\SptShatnAnis(\Sm_k)$.

All the functors in \eqref{eq:intro:Rigidity:SptPAetSmkSptSAetSmkSptSetEtk}
preserve $t$-structures by
hypercomplete \'etale stable connectivity theorem,
\Cref{th:connectivity}(a) of \Cref{cor:etalefraedmotivesformula:Fr},
discussed in \Cref{sect:FrMotConnect}.
This allows us
to define
the realization functor
for
grouplike $\infty$-commutative monoids
\begin{equation}\label{eq:intro:def_RNishypetcompln_CMongp}
    \Re^{\unsmotpref,\Nis}_{\hypet,\compln} \colon
    \CMon^{\gp}_{\Ao,\nis}(\Sm_k)\xrightarrow{L^{\unsmotpref,\nis}_{\et,\compln}}
    \CMon^{\gp}_{\Ao,\et,\compln}(\Sm_k)\xrightarrow{(\Pconst^{\et,\compln}_{\unsmotpref,\et,\compln})^{-1}}
    \CMon^\gp_{\et,\compln}(\Et_k),
\end{equation}
see \Cref{def:Renisetcompln}
or \eqref{eq:intro:Rigidity:connective:SptPAetSmkSptSAetSmkSptSetEtk} below.
The construction provides the following result.

\begin{theorem*}[\protect{\cite[Theorem 7.1]{BachmannRigidity}} and \Cref{th:Remonoidadjoint}]
There are
symmetric monoidal functors
$\Re^{\motpref,\Nis}_{\et,\compln}$
and
$\Re^{\unsmotpref,\Nis}_{\et,\compln}$
of the form
\eqref{eq:intro:def_RNishypetcompln}
and
\eqref{eq:intro:def_RNishypetcompln_CMongp}
and
adjunctions of functors
\begin{equation}\label{eq:ReNishypetcomplndashvPconstcompln}
    \Re^{*,\Nis}_{\hypet,\compln}\dashv \Pconst^{\et,\compln}_{*,\nis},
\end{equation}
where
$*=(\motpref)$ or $*=(\unsmotpref)$.
\end{theorem*}
\begin{remark}Adjunctions
\eqref{eq:ReNishypetcomplndashvPconstcompln}
imply that
functors $\Re^{*,\Nis}_{\hypet,\compln}$
for $k=\bbC$
are equivalent to the $l$-completions of
Betti realizations.
\end{remark}

\subsection{Framed motives and connectivity}\label{sect:FrMotConnect}

To
prove
the mentioned connectivity theorem for
$\SptSGhatnAhypet(\Sm_k)$
and
$\SptShatnAhypet(\Sm_k)$,
and
at the same time,
to study the left-side adjunction
in \eqref{eq:SptPANis}
up to $l$-completion
\begin{equation}\label{eq:adj:LEGmAetniscompln}
L^{\motpref,\nis,\compln}_{\et}\dashv E^{\et}_{\motpref,\nis,\compln}\
\end{equation}
we use
computational machinery
called framed motives theory.

Although the study of $\SptSGAnis(\Sm_k)$ in diagram \eqref{eq:SptPANis}
is immensely complicated, the machinery of Garkusha-Panin framed motives \cite{garkusha2018framed}
leads to
explicit computations for fibrant
resolutions as well as for $\infty$-loop spaces and spectra
\footnote{
More precisely, the framed motives machinery expresses
the Nisnevich stalks of the stable motivic homotopy
sheaves
in terms of Voevodsky's framed correspondences \cite{VoevNotesFr}.
It 
plays the same role for $\SptSGAnis(\Sm_S)$ as Voevodsky's theory \cite{Voe-hty-inv,Voe-motives} for $\DM(k)$. In turn, for the category of \'etale motives with $\bbZn$-coefficients, such a theory is provided for perfect fields
by
\cite{SusVoev,Voe-motives,MVW},
and the inner results and arguments of \cite{Suslin_nonperfectivcahrkmotcomplexes} could be applied to extend the generality.
In this paper, we achieve the same theory for $\SptSGhatnAhypet(k)$.
}
\footnote{
\cite{garkusha2018framed} provides the result for infinite perfect fields,
finite fields are covered independently by
\cite{DrKyllfinFrpi00,elmanto2021motivic},
and arbitrary fields 
are covered by \cite{SHIfr}.
The results over nonperfect fields with $\bbZ[1/p]$-coefficients
are recovered in
\Cref{sect:PurelyinsepextensionsASigma} of this paper.}

The computational results are kept in
the precise terms of Voevodsky's framed correspondences
in parallel to \cite{garkusha2018framed}.
The
arguments
use the $\infty$-category
of tangentially framed correspondences
$\Corr^\fr(k)$
from \cite{elmanto2021motivic}.
The following theorem
collects
$l$-complete hypercomplete \'etale
analogs of
\cite[Th. 10.7, Th 11.1, Th 10.7]{garkusha2018framed},
that
is
a summary of
key results
in
constructing
of
the
framed motives theory.

\begin{theorem*}[\protect{\Cref{cor:etalefraedmotivesformula:Fr}}]
\label{th:intro:Fr}
Let $k$ be a field, and $l\in\bbZ$ be invertible in $k$.
\par\noindent(a)
Consider the composite functor
\[
\Omega^\infty_{\PP^1}\Sigma^\infty_{\PP^1}\colon
\Delta^\mathrm{op}\Sm_k\to \SpcAet(\Sm_k)\xrightarrow{\Sigma^\infty_{\PP^1}} \SpcPetA(\Sm_k)\xrightarrow{\Omega^\infty_{\PP^1}} \SpcAet(\Sm_k)\to \Spc(\Sm_k),
\]
see
Notation
\ref{notation:supersubscriptPPGm:short}.
For any simplicial smooth scheme $Y$ over $k$, and
strictly henselian local essentially smooth $U$ over $k$,
    there is an equivalence in $\Spc$
\begin{equation*}
(\Omega^\infty_{\PP^1}\Sigma^\infty_{\PP^1} Y)(U)_\compln
\xleftarrow{\simeq}
(\Fr(\Delta^\bullet_k\times U,Y)^\gp)_\compln,\end{equation*}
where
the subscript $\compln$ at the left side means the $l$-completion of a grouplike $\infty$-commutative monoid,
see \Cref{subsection:CompletionConservativity},
and the right side is the $l$-completion of the group completion of the special Segal's $\Gamma$-space \eqref{eq:KmapstoFrDeltabulletdashKtimesY}.

\par\noindent(b)
Let $\Sphere$ denote $\PP^{\wedge 1}$ or $\bbG_m^{\wedge 1}$.
Consider the composite functor
\[\Omega^\infty_{\Sphere}\Sigma^\infty_{\Sphere}\Sigma^\infty_{S^1}\colon
\Delta^\mathrm{op}\Sm_k\to \SpcAet(\Sm_k)\xrightarrow{\Sigma^\infty_{\Sphere}\Sigma^\infty_{S^1}} \SptShereetA(\Sm_k)\xrightarrow{\Omega^\infty_{\Sphere}}
\SptSetA(\Sm_k)\to \SptS(\Sm_k),\]
see Notation \ref{notation:supersubscriptPPGm:short}.
For any simplicial smooth scheme $Y$ over $k$, and
strictly henselian local essentially smooth $U$ over $k$,
there is an equivalence in $\SptS$
\begin{equation*}
(\Omega^\infty_{\Sphere}\Sigma^\infty_{\Sphere}\Sigma^\infty_{S^1} Y)(U)_\compln
\xleftarrow{\simeq}
(\Fr(\Delta^\bullet_k\times U,\Sigma^\infty_{S^1}Y)^\gp)_\compln,\end{equation*}
where the right side is given by
an
$\Omega_{S^1}$-spectrum of pointed simplicial sets
in the sense of
\cite{zbMATH01698557}.

\par\noindent(c)
For any $Y\in\Sm_k$, there are natural \'etale local isomorphisms in $\Ab(\Sm_k)$ 
\[
\pi_{i,0}^{\Ao,\et}(\Sigma^\infty_{\PP^1} Y)_l\simeq \pi_i(\Fr(\Delta^\bullet_k\times -,Y))_l,
\quad
\pi_{i,0}^{\Ao,\et}(\Sigma^\infty_{\PP^1} Y)_{/l}\simeq \pi_i(\Fr(\Delta^\bullet_k\times -,Y))_{/l},
\]
for each $i\in\bbZ_{>0}$,
\[
\pi_{0,0}^{\Ao,\et}(\Sigma^\infty_{\PP^1} Y)_l\simeq \ZF(\Delta^\bullet_k\times -,Y)_l,
\quad
\pi_{0,0}^{\Ao,\et}(\Sigma^\infty_{\PP^1} Y)_{/l}\simeq \ZF(\Delta^\bullet_k\times -,Y)_{/l},
\]
and vanishing for $i\in\bbZ_{<0}$,
see Notation \ref{notation:piAtau(F)(U)} and \ref{notation:FnFslashn}.
\end{theorem*}

\begin{remark}\label{rem:Fr1/p}
For $k$ of finite cohomological dimension,
the claim 
of the theorem above 
with $\mathbb Q$-coefficients follows from \cite[Th. 10.7, Th 11.1, Th 10.7]{garkusha2018framed}
because
the adjunction
\begin{equation*}
L^{\Pmotpref,\nis,\mathbb Q}_{\et}\colon
\SpcPANis(k,\mathbb Q)\stackrel{\simeq}{\leftadjright}\SpcPAet(k,\mathbb Q)
\colon E^{\et}_{\Pmotpref,\nis,\mathbb Q}
\end{equation*}
is an equivalence
since
there is an equivalence $\SpcPANis(k,\mathbb Q)^+\simeq\SpcPAet(k,\mathbb Q)$
by \cite[Theorem 13.2]{Elmanto_2022},
and
the minus part $\SpcPAet(k,\mathbb Q)^-$
vanishes by \cite[Lemma 16.2.19]{Cisinski-Deglise-Triangmixedmotives}.
So
the claim of the above theorem
holds with $\bbZ[1/p]$-coefficients,
summarising $l$-completion and $\mathbb Q$-coefficients.
\end{remark}

\begin{corollary*}[\protect{\Cref{th:connectivity}}]
    Let $k$ be a field, and $l\in\bbZ$ be invertible in $k$.
    \par (a)
    The functors
    \begin{equation*}\label{eq:intro:SigmainftySPconnectivity}
    \SptS_{\et,\compln}(\Sm_k)
    \xrightarrow{L^{\et,\compln}_{\Ao}}
    \SptShatnAhypet(\Sm_k)
    \xrightarrow{\Sigma^\infty_\Gm}
    \SptSGhatnAhypet(\Sm_k)
        \end{equation*}

    preserve
    \hetl connective objects, see Notation \ref{notation:hetl}.

    \par (b) There are equivalences of $\infty$-categories
    \begin{equation}\label{eq:intro:Rigidity:connective:SptPAetSmkSptSAetSmkSptSetEtk}
    \xymatrix{
    \Spt^{\bbG_m^{-1}}_{\Ao,\et\geq0,\compln}(\Sm_k)
    \ar[r]^\simeq
    &
    \SptS_{\Ao,\et\geq0,\compln}(\Sm_k)
    \ar[r]^\simeq\ar[d]^\simeq&
    \SptS_{\et\geq0,\compln}(\Et_k)\ar[d]^\simeq
    \\
     &
    \CMon^\gp_{\Ao,\et,\compln}(\Sm_k)
    \ar[r]^\simeq&
    \CMon^\gp_{\et,\compln}(\Et_k)
     }
    \end{equation}
    where the first row consists of the subcategories of \hetl connective objects, see Notation \ref{notation:subscriptgeq} and \ref{notation:subscriptintersect}.
\end{corollary*}

\begin{remark}\label{rem:Conn_Fr1/p}
    For any base field $k$, the functor
    \begin{equation*}
    \SptS_{\et}(\Sm_k)
    \longrightarrow
    \SpcPAet(\Sm_k)
    \end{equation*}
    preserves
    h. \'et. locally connective objects
    because
    of
    \Cref{cor:etalefraedmotivesformula:Fr}(c)
    and \Cref{rem:Fr1/p},
    and since $\Spc^{\PP^{-1}}_{\Ao,\et,\hat{p}}(\Sm_k)\simeq 0$
   \end{remark}
    We refer to
    \cite[\S 4]{Algebraiccobordismandetalecohomology}
    for
    the background 
    regarding
    the stable motivic connectivity with respect to the \'etale topology.

\subsection{Comparison results}\label{sect:intro:Comparisonresults}

Combining computational miracles of
framed motives with respect to Nisnevich topology
\cite{garkusha2018framed}
and
\Cref{cor:etalefraedmotivesformula:Fr}
discussed in the previous section
we obtain
the following result.

\begin{theorem*}[\protect{\Cref{cor:nisconget}}]
Let $k$ be a field, and $l\in\bbZ$ be invertible in $k$.
\par (a) Let $E\in\SpcPhatnAnis(\Sm_k)$ be effective,
and denote by the same symbol its image in $\SpcPhatnAhypet(\Sm_k)$.
The natural homomorphism of presheaves
\begin{equation*}\label{eq:intro:piA1nissimeqpiA1et}
\pi_{i+j,j}^{\Ao,\nis}(E)
\rightarrow
\pi_{i+j,j}^{\Ao,\et}(E),
\end{equation*}
is an \'etale local isomorphism,
for each $i\in\bbZ$, $j\in\bbZ_{\geq 0}$.
\par (b) The diagram of $\infty$-categories
\[
\xymatrix{
\SptSGhatnAniseff(\Sm_k)\ar[d]^{\Omega^\infty_\Gm}
\ar[r]^{L^{\motpref,\nis,\compln}_{\et}}
&
\SptSGhatnAeteff(\Sm_k)
\ar[d]^{\Omega^\infty_\Gm}
\\
\SptS_\nis(\Sm_k)\ar[r]^{L^{\nis}_\et}
&
\SptS_\et(\Sm_k)
}
\]
is commutative, see Notation \ref{notation:effveff} for the first row.
\end{theorem*}
\begin{remark}
By \Cref{rem:Fr1/p}, 
\Cref{cor:nisconget}
holds with $\bbZ[1/p]$-coefficients
as well.
\end{remark}

Using \Cref{cor:nisconget} in combination with \eqref{eq:intro:Rigidity:SptPAetSmkSptSAetSmkSptSetEtk} and \eqref{eq:intro:Rigidity:connective:SptPAetSmkSptSAetSmkSptSetEtk}
we get the comparison theorem of the realization functors \eqref{eq:intro:def_RNishypetcompln} and \eqref{eq:intro:def_RNishypetcompln_CMongp}.

\begin{theorem*}[\protect{\Cref{th:piAniswnBetenhypetANDLoopspectraandspacesANDRealisationfunctors}}]
Let $k$ be a field, and $l\in\bbZ$ be invertible in $k$.

(a)
    For any $E\in \SptSGAniseff(\Sm_k)$    there is
    a natural equivalence in $\SptShatnhypet(\Et_k)$
    \[\Gamma^{\motpref,\nis}_{\et,\compln}(E)\simeq
    \Re^{\motpref,\Nis}_{\hypet,\compln}(E),\]
    where
    $\Gamma^{\motpref,\nis}_{\et,\compln}=L^{\et}_{\et,\compln}\Gamma^{\motpref,\nis}_{\et}$,
    see Notation \ref{notation:functorsGammas}.

(b)
    The following diagram of $\infty$-categories is commutative
    \begin{equation*}\label{eq:intro:RealisLoopFunctorsCommute}\xymatrix{
    \SpcPAniseff(\Sm_k)\ar[r]^{\Re^{\Pmotpref,\nis}_{\et,\compln}}\ar[d]^{\Omega^\infty_{\PP^1}}&
    \Spt_{\et,\compln}(\Et_k)\ar[d]^{\Omega^\infty_{S^1}}
    \\
    \CMon^\gp_{\Ao,\nis}(\Sm_k)\ar[r]^{\Re^{\unsmotpref,\nis}_{\et,\compln}}\ar@{-->}[d]|{k=\bbC}&
    \CMon^\gp_{\et,\compln}(\Et_k)\ar@{-->}[d]|{k=\bbC}
    \\
    \Spc_{\Ao,\nis}(\Sm_\bbC)\ar[r]^{Be}&
    \Spc
    ,}\end{equation*}
      where
    the upper two rows are considered for any field $k$,
      and
    the bottom row is considered for $k=\bbC$,
    where $Be$ denotes the Betti realization, and the vertical functors are induced by the functor $\CMon^\gp\to\Spc$.
\end{theorem*}

Point (a) of \Cref{th:piAniswnBetenhypetANDLoopspectraandspacesANDRealisationfunctors} above
for algebraically closed base field $k=\overline{k}$,
recovers the result by Zargar \cite{zbMATH07103864},
regarding pro-spectra \'etale realization functors $\Sm_{\overline k}\to \mathrm{Pro}(\SptS)$, $X\mapsto \Pi^\et_\infty X$,
that, in its turn, generalizes
the comparison theorem by Levine \cite{Levine_2013},
regarding
the Betti realization over the base field $\bbC$.

\begin{remark}\label{rem:basechangefieldextensionArgumentcompariosn}
    We would like to mention that \Cref{th:piAniswnBetenhypetANDLoopspectraandspacesANDRealisationfunctors}(a) is
    stronger than the 
    comparison theorem from \cite{zargar2017comparison} because it covers all base fields,
    while \cite{zargar2017comparison} covers only algebraically closed ones.
       We might suspect that there is an alternative way to
     deduce our result from \cite{zargar2017comparison} 
     using properties of base change functors
     for separable base field extensions like in \Cref{lm:commute:els:Let_AND_:LA_AND_OmegaGm}
     and purely unseparable ones provided by anyone of 
     \cite{Elmanto2018PerfectionIM} 
     or 
     \Cref{prop:pinsepextSchchk}, 
     but such a proof is not pursued here.
    However, we would like to stress that our proof is formally independent of any one of \cite{zargar2017comparison} and \cite{Elmanto2018PerfectionIM},
    which we will discuss
       more in detail
    in \Cref{sect:ComparisonoftheapproachesforComparisonTheorems} and \Cref{sect:ProofStrategy}
    respectively.
\end{remark}

\subsection{Comparison of the approaches for \ComparisonTheorems}
\label{sect:ComparisonoftheapproachesforComparisonTheorems}

In the proof of
\begin{equation}\label{eq:piAnisizeroFcomplnbbCsimeqpiiBettiFbbCcompln}
\pi_{i,0}^{\Ao,\nis}(E_\compln)(\bbC)\cong \pi_i(\Betti(E)_\compln),
\end{equation}
\cite[Theorem 3]{Levine_2013}, Levine
uses the slice tower of the motivic sphere spectrum
$\mathbbm{1}_k$
and its Betti realization, giving rise to the slice and Adams spectral sequences, respectively. Based on results concerning information on the layers or "slices" in the two filtrations, the first page of these spectral sequences consists of motivic homology and singular homology, respectively. In this way, modulo several technical details
\eqref{eq:piAnisizeroFcomplnbbCsimeqpiiBettiFbbCcompln} was reduced
to the foundational isomorphism
\begin{equation}\label{eq:HsingXZnsimeqHXCZn}
H^\mathrm{sing}_*(X,\bbZn)
\cong
H_*(X(\bbC),\bbZn), \quad
X\in\Sm_\bbC,\, l\in\bbZ,
\end{equation}
proven by Suslin and Voevodsky \cite{SusVoev}.
Here
\[H^\mathrm{sing}_*(X,\bbZn)\cong
\mathrm{Hom}_{\DM(k)}(X,\bbZn)\cong
\mathrm{Hom}_{\DM_\cet(k,\bbZn)}(X,\bbZn)
\]
are the motivic cohomologies
concerning
Voevodsky's triangulated category of motives \cite{Voe-motives,MVW}.
\cite[Part 2]{MVW}
provides the result like \eqref{eq:HsingXZnsimeqHXCZn}
regarding \'etale realization like \eqref{eq:intro:def_RNishypetcompln}
for a perfect field $k$.
Zargar's proof of the analog of \eqref{eq:piAnisizeroFcomplnbbCsimeqpiiBettiFbbCcompln}, see \cite{zargar2017comparison},
deduces the result
comparing
spectral sequences
related to 
the scheme $X\in\Sm_{\overline{k}}$
and
the pro-space
$\Pi^\et_\infty X$.

It was
pointed out
by Panin
that
there should be a way to recover
\eqref{eq:piAnisizeroFcomplnbbCsimeqpiiBettiFbbCcompln}
by a direct argument following the strategy of \cite{SusVoev} for \eqref{eq:HsingXZnsimeqHXCZn}
using
framed correspondences
and its associated
\RigidityTheorem
\cite{FramedRigidityLoc}
instead of $\Cor$-correspondences
and the \RigidityTheorem from
\cite[\S 4]{SusVoev}.
Though the present article uses framed correspondences,
it does not provide such an argument because of two reasons:
\par
(1)
Our proof
does not follow the strategy of \cite{SusVoev},
but rather
providing alternative
proofs for
both comparison isomorphisms
\eqref{eq:HsingXZnsimeqHXCZn}, \eqref{eq:piAnisizeroFcomplnbbCsimeqpiiBettiFbbCcompln}.
The strategy of \cite{SusVoev} used
the analogy of
the Dold-Thom Theorem in algebraic topology
and
the representability of
presheaves of
$\Cor(-,Y)$ on $\Sm_k$
by the symmetric powers
$\mathrm{Sym}^\infty Y$.
For this reason, Panin suggested to
introduce a complex analytic analog
of framed correspondences and
compare it with
the Betti realizations of the geometric models of framed correspondences
from
\cite{elmanto2021motivic} or
\cite{SmModelSpectrumTP}.
This is a serious and complicated direction of the study,
which is not pursued in the present article.

Applying our argument over $\bbC$
we would consider Betti realizations of suspension spectra
\[\Betti(\Sigma^\infty_{\PP^1}Y)\simeq\Sigma^\infty_{S^2}Y(\bbC)\]
only
instead of
the Betti realizations of
the framed motives
geometric models
or $\mathrm{Sym}^\infty Y(\bbC)$.
The advantage of framed correspondences
that allows this
is the combination of
the \ReconstructionEquivalence 
$\SH(k)\simeq\SH^\fr(k)$
\cite[\S 13]{garkusha2018framed} or \cite[Theorem 3.5.12]{elmanto2021motivic},
and
the \CancellationTheorem,
see
\Cref{th:GmCancellation:SfrAet}(a) and \Cref{th:SCancellation}.
Note that
the cancellation
was not used
in \cite{SusVoev},
and
the \ReconstructionEquivalence
does not have a direct analog for $\mathbf{DM}(k)$.

(2)
On the other hand,
our proof uses
the
external ingredient, namely,
the \RigidityTheorem by Bachmann \cite{BachmannRigidity,bachmann2021remarks}.

\subsection{Proof strategy, novelties, and ingredients}\label{sect:ProofStrategy}

\par
Although the \RigidityEquivalence
\begin{equation}\label{eq:Rigidity:SptPAcetSmkSptScetEtk}
    \Gamma^{\et,\compln}_{\et,\compln}\colon \SptPhatnAhypet(\Sm_k)\stackrel{\simeq}{\rightadjleft}\SptShatnhypet(\Et_k)\colon \Pconst^{\et,\compln}_{\et,\compln} 
\end{equation}
provided by
the composite equivalence
in
\cite[Theorem 3.1]{bachmann2021remarks}
is not
used in the proof of \Cref{cor:etalefraedmotivesformula:Fr},
the equivalence
\begin{equation}\label{eq:intro:SptSAet(k)wedgensimeqSptPAet(k)wedgen}
\SptShatnAhypet(\Sm_k)\simeq\SptPhatnAhypet(\Sm_k)   
\end{equation} provided by the left side
equivalence in
\cite[Theorem 3.1]{bachmann2021remarks}
is used,
being
applied to the $\Gm$-\Cancellation, see \Cref{th:GmCancellation:SfrAet}(a).
\begin{remark}[Alternative strategy]
Let us note
that up to some additional steps of the argument
it would be enough to use
the invertablility of $\bbG_m^{\wedge 1}$
in the $\Ao$-localization of the derived category of \'etale
$\bbZn$-modules $\mathbf D^-_{\cet}(\Sm_k,\bbZn)$,
i.e.
the equivalence
\[\mathbf D^-_{\cet}(\Sm_k,\bbZn)\simeq
\mathbf D^-_{\cet}(\Sm_k,\bbZn)[\bbG_m^{\wedge 1}]\]
instead of \eqref{eq:intro:SptSAet(k)wedgensimeqSptPAet(k)wedgen}.
Moreover, in this way, framed motives techniques would conversely
simplify the proof of
the \RigidityEquivalence
\eqref{eq:Rigidity:SptPAcetSmkSptScetEtk},
providing
a reduction
of
\eqref{eq:Rigidity:SptPAcetSmkSptScetEtk}
to the categories
$\mathbf{DA}_{\cet}(-,\bbZn)$
studied in \cite{etalerealizationAyoub}.
We do not pursue this strategy,
because we could not say that such an argument for

\eqref{eq:Rigidity:SptPAcetSmkSptScetEtk}
would be completely independent from
the strategies of \cite{BachmannRigidity} and \cite{etalerealizationAyoub},
while the text of our article would be longer.
\end{remark}

The reason for using the \emph{hypercomplete} étale topology is in our arguments
of the proof of the equivalence
\begin{equation}\label{eq:LfrAetLfrALfret}\calL^{\fr}_{\Ao,\et} \simeq \calL^\fr_{\et}\calL^{\fr}_{\Ao}\colon \Sptfr_{\Sigma,\compln}(\Sm_k)\to\Sptfr_{\Sigma,\compln}(\Sm_k),
\end{equation}
see \Cref{th:LA1et},
that
is deduced from
the \SHITheorem, see \Cref{th:SHI:ZPShfrA},
using the local-to-global
spectral sequence 
\cite[Proposition 2.13]{ClausenMathewHypdescetaleKth}, see \Cref{prop:specseq_for_etalepisheaves:perfect:boundedncomplete}.

For perfect base fields,
like as in \cite[Part 2]{MVW},
the mentioned
strict homotopy invariance theorem
is
provided by
\cite[Lemma 9.23]{MVW}, \cite[XV 2.2]{SGAfour}
and the \RigidityTheorem in this context
\cite{FramedRigidityLoc}.

The perfectness assumption
in Voevodsky's theory on $\DM(k)$
in the Nisnevich setting with $\mathbb Z[1/p]$-coefficients
was eliminated
by Suslin \cite{Sus-nonperftalk,Suslin_nonperfectivcahrkmotcomplexes}.
While there is an argument that covered Nisnevich local results with $\bbZ$-coefficients  \cite{SHIfr},
since in our setting $(n,\echark k)=1$,
the principles of the argument in \cite{Suslin_nonperfectivcahrkmotcomplexes}
are enough for our purposes.
The argument in \cite{Suslin_nonperfectivcahrkmotcomplexes}
implicitly proves fully faithfulness of the base change functor
\begin{equation}\label{eq:SptHZCorkKfullyfatihful}
    \SptHZ_{\Ao,\Sigma}(\Cor_k)[1/p]\hookrightarrow\SptHZ_{\Ao,\Sigma}(\Cor_{k^\mathrm{perf}})[1/p]
\end{equation}
where $k^\mathrm{perf}$ is
the purely inseparable closure
of $k$.
Note, on the other hand, there is
the equivalence
\begin{equation}\label{cor:th:pinsepclosure:gpfr:Equalivalenceinvchark}
e^*\colon \PShfrgpAnis(\Sm_k)[1/p]\stackrel{\simeq}{\rightadjleft} \PShfrgpAnis(\Sm_{k^\mathrm{perf}})[1/p]\colon e_*
\end{equation}
proven by
Elmanto and Khan, see \cite[Theorem 2.1.1]{Elmanto2018PerfectionIM} for $S=\Spec k$.
Nevertheless,
\eqref{cor:th:pinsepclosure:gpfr:Equalivalenceinvchark} is not applicable to results as
the \SHITheorem,
or \eqref{eq:LfrAetLfrALfret}, or \eqref{eq:intro:OmegaGmSigmaPYUcomplnsimeqFrDeltaUSigmaSYgpcompln:strictlochenseesssmschemeU},
because
it implies only a motivic equivalence
and not a
schemewise one.

In \Cref{sect:PurelyinsepextensionsASigma}, \Cref{prop:pinsepextSchchk},
we prove the equivalence
\[\PShfrgpAS(\Sch_k)[1/p]\stackrel{\simeq}{\leftadjright} \PShfrgpAS(\Sch_{k^\mathrm{perf}})[1/p],\]
that induces the coreflection
\[\Spc^\gpfr_{\Ao,\Sigma}(\Sm_k)[1/p]\leftadjright\Spc^\gpfr_{\Ao,\Sigma}(\Sm_{k^\mathrm{perf}})[1/p].\]
This allows to reduce \eqref{eq:LfrAetLfrALfret} with $\bbZ[1/p]$-coefficients for any localization $\calL_\tau$ that
commutes with the presheaf direct image functors along $\Sm_k\to\Sm_K$, and $\Sm_k\to\Corr^\fr_k$,
satisfied for $\calL_\et$ by \Cref{lm:commute:els:Let_AND_:LA_AND_OmegaGm}(1).

The novelty of our proof strategy
for the
results on $\infty$-loop spaces and spectra, see
\Cref{lm:etalefraedmotivicloopsspaces},
is
concentrated in \Cref{sect:SCancelation},
and
relates
to
the equivalence
\[
\Omega_{S^1,\Aet}^{\infty,\fr}\Sigma^{\infty,\fr}_{S^1,\Aet}
\simeq \calL_{\et}\calL_{\Ao},\colon
\Spc^\frgp_{\Sigma,\compln}(\Sm_k)\to\Spc^\frgp_{\Sigma,\compln}(\Sm_k),\]
where $\Spc^\frgp_{\Sigma}(\Sm_k)\simeq\mathrm{Func}^\times(\Corr^\fr,\Spc)^\gp$,
and see
Notation \ref{notation:supersubscriptPPGm:short}
for $\Omega_{S^1,\Aet}^{\infty,\fr}\Sigma^{\infty,\fr}_{S^1,\Aet}$.

There are also novelties regarding
the arguments
for the commutativity like $\calL_\et\gamma_*\simeq\gamma_*\calL^\fr_\et$.
In the context of the \'etale and Nisnevich localizations in \cite{VoevNotesFr,hty-inv,elmanto2021motivic}, the reasoning
is based on \v{C}ech constructions,
while the hypercompletion required
us to replace it by
an argument that uses points of the topology only, see \Cref{lm:Letgammads}(1).
A similar argument was
written earlier
for Nisnevich setting in the context of $\mathrm{GW}$-correspondences \cite{GWDMeff}.

Lastly, we mention that much thought has been spent establishing notation, to which we dedicate an entire Section \ref{notation}. Hopefully, this will make the mathematical content of this article clearer, more condensed, and more elegant. Novelties in notation include that for presheaf $\infty$-categories $\calSP^*_\star(\mathcal{C})$ we are consistent in letting superscripts indicating structure and subscripts indicating conditions, see \ref{notation:Psstardash_lists}. Nevertheless, we want to reassure the reader that almost all notations are already standard in the literature.

\subsection{Acknowledgement}



The authors are thankful to all people they have consulted
during the long time collaboration for more than three and a half years, 
and especially we are thankful for the following consultations.
\begin{itemize}
\item 
The second author would like to express his gratitude to Bachmann for patient supervision on topics related to this article, and to Østvær for concise and insightful mathematical guidance and other extremely appreciated help. 
\item
The first author is very thankful to Panin for 
discussions on ways of recovering of Levine's comparison theorem using framed correspondences and 
his invitation to collaboration on the study of their Betti realizations in the early beginning.
\end{itemize}
The research is supported by Russian Science Foundation grant 19-71-30002
provided its special support for 
\Cref{sect:Framed,sect:PurelyinsepextensionsASigma,sect:LAet-SptSfrAet,sect:OmegaPAetRNisetcompln}, 
and 
Theorems \ref{cor:nisconget}(a) and \ref{th:piAniswnBetenhypetANDLoopspectraandspacesANDRealisationfunctors}(a)
according to the grant conditions.
The second author would also like to acknowledge the support of EMOHO-project $112403001$ for his trips during the research.

\subsection{Notation and convention}
\label{notation}


We apologize for the use of nonstandard notation for the $\infty$-categories
often considered in the motivic homotopy theory.
To illustrate our notation 
and help reader 
we write two following examples 
relating our notation to the ones in 
the cited sources.
The motivic localization and 
$\infty$-suspension and $\infty$-loop functors
denoted in \cite{Elmanto_2022}
\[
\mathrm{PSh}(\Sm_k)
\xrightarrow{L_{mot}} 
\mathbf{H}(k)
\xrightarrow{\Sigma^\infty_{\PP^1}} 
\mathbf{SH}(k)\xrightarrow{\Omega^\infty_{\PP^1}} 
\mathbf{H}(k)\to
\mathrm{PSh}(\Sm_k)
\]
are denoted in our article
\[
\Spc(\Sm_k)\xrightarrow{L_{\Ao,\nis}} 
\Spc_{\Ao,\nis}(\Sm_k)\xrightarrow{\Sigma^\infty_{\PP^1}} 
\SpcPANis(\Sm_k)\xrightarrow{\Omega^\infty_{\PP^1}} 
\Spc_{\Ao,\nis}(\Sm_k)\xrightarrow{E_{\Ao,\nis}} 
\Spc(\Sm_k).
\]
The equivalences 
\begin{equation*}
    \mathcal{SH}_{\acute{e}t}(k)^\wedge_l
    \simeq
    \mathcal{SH}^{S^1}_{\acute{e}t}(k)^\wedge_l
    \simeq
    \mathrm{Sp}(k_{\acute{e}t}^\wedge)^\wedge_l
\end{equation*} 
from
\cite[Theorem 3.1]{bachmann2021remarks}
are denoted in our article
\begin{equation*}
    \SptPhatnAhypet(\Sm_k)
    \simeq
    \SptShatnAhypet(\Sm_k)
    \simeq
    \SptShatnhypet(\Et_k)
.\end{equation*} 
We use these unusual notation 
because of a general pattern, see Notation \ref{notation:Psstardash_lists}, 
that is suitable for 
other 
$\infty$-categories
considered in \Cref{sect:HypcEtFrMotives}
of our text. 
\begin{enumerate}

\item 
Throughout the text we suppose that $S$ is a qcqs scheme.

\item \label{notation:kcahark=p}
$k$ denotes a field,
and
$\echark k$ is the exponential characteristic of $k$.

\item 
$l$ denotes an integer 
such that $(n,\echark k)=1$ if otherwise is not mentioned.

\item
$F_\compln=F^\wedge_l$ is the $l$-completion of $F$;
$F^*_{\hat{l}}=(F^*)^\wedge_l$, for a given $F^*$ with superscript $*$. Example: $F^\gp_\compln=(F^\gp)^\wedge_l$.

\item \label{notation:calPcompln}
$L_{\compln}\colon\calSP\to\calSP^\wedge_l=\calSP_\compln$
is the left adjoint to the embedding of the $l$-completion of a prestable presentable symmetric monoidal $\infty$-category $\calSP$
that we denote $\calSP^\wedge_l$ or $\calSP_\compln$.

\item \label{notation:Fn=cofibFnF}
$F/l=\cofib(F\to F)$ for a 
given morphism $n\colon F\to F$ in a cocomplete $\infty$-category.

\item \label{notation:FnFslashn}
$F_l=\Ker(F\xrightarrow{n}F)$, 
$F_{/l}=\Coker(F\xrightarrow{n}F)$, 
for $n\colon F\to F$ in an abelian category.

\item 
$F^D=(-)=F(D\times-)\in\calSP(\calC)$
for $D\in\calC$, $F\in\calSP(\calC)$. 
Example: 
$F^{\Ao}(-)=F(\Ao\times-)$.

\item\label{notation:SpcSptSHtop}
$\Spc$, 
$\Spcpointed$, 
and 
$\Spt$ 
are the $\infty$-categories of spaces,
pointed spaces, 
and spectra.\\
$\SH^\mathrm{top}$ is the homotopy category of $\Spt$.

\item\label{notation:CMonCMongp}
$\CMon$ is the $\infty$-category of $\infty$-commutative monoids, i.e. $\mathbb{E}_\infty$-monoids.
$\CMon^\gp$ is the subcategory of grouplike objects.

\item\label{notation:SptZ}
${\SptHZ}$ is the $\infty$-category of modules over the Eilenberg–MacLane $S^1$-spectrum $\HZ$.\\
${\SptHZn}$ is the $\infty$-category of modules over the Eilenberg–MacLane $S^1$-spectrum $\HZn$.

\item
$\Ab=\mathrm{Mod}_{\bbZ}$ is the category of abelian groups.
$\ZnPSh=\mathrm{Mod}_{\bbZn}$.

\item \label{notation:DAb}
$\mathbf D(\mathrm{Ab})$ is the derived category of $\Ab$, i.e. 
the homotopy category of the $\infty$-category $\SptHZ$.

\item
$\Delta^\mathrm{op}\mathcal C=\Func(\Delta^\mathrm{op},\mathcal C)$.

\item
$\Sch_S$, $\Sm_S$, $\Et_S$, $\SmAff_S$ 
are the categories of $S$-schemes and smooth, \'etale,
affine smooth $S$-schemes, respectively. 
\item
$\mathrm{EssSm}_S$
denotes the category of essentially smooth $S$-schemes,
which are
filtered limits
of smooth ones with affine \'etale transition morphisms.

\item\label{notation:calP(calC)} 
$\calSP(\mathcal C)=\Func(\mathcal C^\mathrm{op},\calSP)$
for 
a
given $\infty$-category $\calSP$, and $\infty$-category $\mathcal C$.

\item
$\calSP_{\Sigma}(\mathcal C)$ is the subcategory spanned by the radditive presehaves.



\item
$\calSP(S)=\calSP(\Sm_S)$,
$F\simeq_{\Sm_S} G$ means $F\simeq G\in \calSP(\Sm_S)$.

\item\label{notation:Psstardash_lists}
$\calSP^*_\star(-)$ 
in what follows are 
the $\infty$-categories
with superindices indicating structure and subindices indicating conditions:\\
$*=
S^1,{\PP^1},\Gm,\fr,\gp,\gpfr,\bbZ,\bbZ/l,\bbZ\fr,\bbZ\fr/l$,
here
$\bbZ\fr/l$ 
equals
the combination $\bbZ/l,\fr$,\\
$\star=
\compln,\Ao,\Sigma,\Nis,\cet,\et,
\Nisgeq,\etgeq,
\Nisleq,\etleq,
\etbdleqgeq,$
and their combinations.


\item\label{notation:L}
For any $*$ from the list above and $\star=\compln,\Ao,\Sigma,\Nis,\cet,\et$, we have functors\\
$L^{*,\star_0}_{\star_1}\colon \calSP^*_{\star_0}(-)\leftadjright\calSP^*_{\star_1}(-)\colon E_{\star_0}^{*,\star_1}, 
\quad
L^{*,\star_0}_{\star_1}\dashv E_{\star_1}^{*,\star_0}$,
\\
$\Str^{\star,*_0}_{*_1}\colon \calSP^{*_0}_\star(-)\leftadjright\calSP^{*_1}_\star(-)\colon \Fog^{\star,*_1}_{*_0}, \quad\Str^{\star,*_0}_{*_1}\dashv \Fog_{\star,*_0}^{*_1}$,\\
$\calL^{*,\star_0}_{\star_1}= E_{\star_0}^{*,\star_1} L^{*,\star_0}_{\star_1}\colon\calSP^*_{\star_0}(-)\to\calSP^*_{\star_0}(-)$,
\\
$\calS^{\star,*_0}_{*_1}= \Fog_{\star,*_0}^{*_1} \Str^{\star,*_0}_{*_1}\colon\calSP^{*_0}_\star(-)\to\calSP^{*_0}_\star(-)$.
\par
\noindent
To help the reader remember we have chosen the letters with some intention:
$L$-localization, $E$-embedding, $S$-structure and $F$-forget.
\par
We let us write in short 
$L^{*,\star_0}_{\star_+}$ for $L^{*,\star_0}_{\star_0\cup\star_+}$,
or even write
$L_{\star_+}$
, when $*$ and $\star_0$ are defined by the context.


\item\label{notation:superscriptsubscript}
$\calSP^*_{\star}(-)=
\calSP^*(-)\times_{\calSP(-)}\calSP_\star(-)\simeq
\Fog_*^{-1}(\calSP_\star(-))$. 
Examples: $\Spt^{\PP^{-1}}_\star(-)$, $\Spt^\fr_\star(-)$, $\SptHZ_\star(-)$.


\item\label{notation:subscriptintersect}
$\calSP^*_{\star_1,\star_2}(-)=\calSP^*_{\star_1}(-)\cap\calSP^*_{\star_2}(-)$.
Examples: \\
$\calSP_{\Ao,\Nis}(-)=\calSP_{\Ao}(-)\cap\calSP_{\Nis}(-)$, \\
$\calSP^{*}_{\star,\compln}(-)\simeq\calSP^{*}_\star(-)^\wedge_l$, 
that is $(\calSP^\wedge_l)^{*}_\star(-)$ for prestable $\calSP$,
see
Notation \ref{notation:calPcompln},\\
$\Spt^{\PP^{-1}}_{\Ao,\Nisgeql{0}}(\Sm_S)=\Spt^{\PP^{-1}}_{\Ao}(\Sm_S)\cap\Spt^{\PP^{-1}}_{\Nisgeql{0}}(\Sm_S)$,
see Notation \ref{notation:subscriptgeq}.
\item\label{notation:calPsonestwostar}
$\calSP^{*_1,*_2}_\star(-)=(\calSP^{*_1})^{*_2}_\star(-)$, for $*_1=\bbZ,\bbZ/l,\geq$.
Example:\\ 
$\Spt^{\PP^{-1},\bbZ\fr/l}_{\Ao,\Nis}(\Sm_S)=(\Spt^{\bbZ/l})^{\PP^{-1}}_{\Ao,\Nis}(\Corr^\fr_S),$
which is equivalent to the subcategory of $\Ao$-invariant Nisnevich local objects in
$\Func\left((\Corr^\fr_S)^\mathrm{op},\SptHZn)\right)\left[\PP^{\wedge -1}_S\right]$,
see Notation
\ref{notation:S1Pwedge1Gmwedge1}.

\item\label{notation:gammaSmSCorrfrS} 
$\gamma\colon \Sch_S\to \Corr^\fr(\Sch_S)$, see \cite[\S 4]{elmanto2021motivic} for  
$\Corr^\fr(\Sch_S)$,
then \\$\gamma^*=\Str_\fr$, $\gamma_*=\Fog_\fr$,
see Notation \ref{notation:L}.
\item\label{notation:CorrfrS}
$\Corr^\fr_S=\Corr^\fr(S)=\Corr^{\fr}(\Sm_S)$ is the subcategory spanned by smooth $S$-schemes.
\item \label{notation:Corrgpfr}
$\Corr^\gpfr(-)$
is the additivisation of the semiadditive category $\Corr^\fr(-)$,\\
$\gamma_\gp\colon\Sch_S\to\Corr^\gpfr(\Sch_S)$,\\
$\gamma^*_\gp=\Str_{\gpfr}$, $\gamma^\gp_*=\Fog_{\gpfr}$,
see Notation \ref{notation:L}.
\item 
$h^\fr(-)=\Corr^\fr(-,Y)$,
$h^\gpfr(-)=\Corr^\gpfr(-,Y)$.


\item\label{notation:calPfr(calC)}
$\mathcal P^\fr(S)=\mathcal P(\Corr^\fr_S)$,
$\mathcal P^\gpfr(S)=\mathcal P(\Corr^\gpfr_S)$,\\
$S^\fr_\gpfr\colon\mathcal P^\fr(\Sm_S)\to \mathcal P^\gpfr(\Sm_S)$
is the inverse image functor. 
\par\noindent
We let us
freely use 
equivalences 
$\mathcal P^\gpfr(\Sm_S)\simeq\mathcal P^\fr(\Sm_S)^\gp$, 
$S_\gpfr^\fr(F)\simeq F^\gp$,
where the right sides are 
subcategory
of
the group complete objects in
$\mathcal P^\fr(\Sm_S)$,
and
the group completion
of
$F$.


\item\label{notation:gammaFr}
$\gamma_{\Fr_+}\colon\Sm_S\to\Fr_+(S)$;
$\gamma_{\Fr_{+}}^*$ and $\gamma^{\Fr_{+}}_*$ are the inverse and direct images .

\item \label{eq:ovZF}
$\ovZF(X,Y)=\operatorname{coker}(\ZF(X\times\Ao,Y)\stackrel{0,1}{\rightrightarrows}\ZF(X,Y))$,
\cite[Def. 2.11]{hty-inv}
for $\ZF(-,Y)$.







\item \label{notation:SptHZgpfr}
${\SptHZfr_\Sigma}(\Sm_S)={\SptHZ_\Sigma}(\Corr^\fr(\Sm_S))$, and
${\SptHZnfr_\Sigma}(\Sm_S)={\SptHZn_\Sigma}(\Corr^\fr(\Sm_S))$,\\
$\Spc^\Zgpfr_{\Sigma}(-)=(\Spt^{\bbZ}_\geq)_{\Sigma}(\Corr^\gpfr(-))$,
see Notation \ref{notation:calPsonestwostar} for universal pattern. 




%
%
%


\item \label{notation:S1Pwedge1Gmwedge1}
$S^1 = \cofib(\partial\Delta^1\to \Delta^1)\in\Spc$,\\
$\PP^{\wedge 1}_S=\cofib(\pt_k\xrightarrow{\infty} \PP^1_S)\in\Spc(\Sm_S)$,\\
$\Gmwo=\cofib(\pt_k\xrightarrow{1}\Gm)\in\Spc(\Sm_S)$.


\item\label{notation:superscriptPPGmcompln}
$\calSP^{\mathbb S^{-1},*}_\star(-)=\calSP^{*,\mathbb S^{-1}}_\star(-)=\calSP^*_\star(-)[\mathbb S^{\wedge -1}]$, 
where $\mathbb S=\PP^{\wedge 1},\bbG_m^{\wedge 1},S^1$. 
Examples:\\
$\Spt^{\PP^{-1},*}_\star(-)=\Spt^*_\star(-)[(\PP^1)^{\wedge -1}]\simeq \Spc^{\PP^{-1},*}_\star(-)$,\\
$\Spt^{\bbG_m^{-1},*}_\star(-)=\Spt^*_\star(-)[\bbG_m^{\wedge -1}]$,
$\Spc^{\bbG_m^{-1},*}_\star(-)=\Spc^*_\star(-)[\bbG_m^{\wedge -1}]$,\\
$\Spc^{S^{-1}}(\calC)\simeq\Spt(\calC)\simeq \Spt^{S^{-1}}(\calC)$.

\item 
$\Sigma_{\mathcal E}^*\colon \calSP^*(-)\leftadjright\calSP^*(-)\colon\Omega_{\mathcal E}^*$,
$\Sigma_{\mathcal E}^{\infty,*}\colon \calSP^*(-)\leftadjright\calSP^{*,\mathcal E}(-)\colon\Omega_{\mathcal E}^{\infty,*}$.
Example:\\
$\Sigma_{S^1}^*$ and $\Omega_{S^1}^*$ are schemewise.
\item 
We let us short notation $\Sigma^{\infty,\gpfr,\Sigma}_{\mathcal E}$ to $\Sigma^{\infty,\gpfr}_{\mathcal E}$,
while $\Sigma^{\infty,\gpfr}_{\mathcal E}$ does not always preserves radditive presheaves.

\item \label{notation:supersubscriptcombPPGm}
$\Sigma_{\mathcal E}^{*,\star}\colon \calSP^*_{\star}(-)\leftadjright\calSP^{*}_{\star}(-)\colon\Omega_{\mathcal E}^{*,\star}$,
$\Sigma_{\mathcal E}^{\infty,*,\star}\colon \calSP^*_{\star}(-)\leftadjright\calSP^{*,\mathcal E}_{\star}(-)\colon\Omega_{\mathcal E}^{\infty,*,\star}$.
Example:\\
$\Sigma_{S^1}^{*,\Ao,\et} \dashv\Omega_{S^1}^{*,\Ao,\et}$.

\item \label{notation:supersubscriptPPGm:short}
$\Sigma_{\mathcal E,\star}^{*}\colon \calSP^*(-)\leftadjright\calSP^{*}_{\star}(-)\colon\Omega_{\mathcal E,\star}^{*}$,
$\Sigma_{\mathcal E,\star}^{\infty,*}\colon \calSP^*(-)\leftadjright\calSP^{*,\mathcal E}_{\star}(-)\colon\Omega_{\mathcal E,\star}^{\infty,*}$.
Example:\\
$\Sigma^{\infty}_{S^1,\Ao,\et}\simeq\Sigma^{\infty,\Ao,\et}_{S^1}L_{\Ao,\et}$ and $\Omega^\infty_{S^1,\Ao,\et}\simeq E_{\Ao,\et}\Omega_{S^1}^{\infty,\Ao,\et}$.
\\
We let us short notation 
$\Sigma^{\infty,*}_{\mathcal E,\star}$ and $\Omega^{\infty,*}_{\mathcal E,\star}$ to
$\Sigma^{\infty,*}_{\mathcal E}$ and $\Omega^{\infty,*}_{\mathcal E}$
when 
the context
indicates
subscript $\star$ for the 
target 
and 
source 
categories of the above functors
respectively.


\item
$L_I\colon \calSP(\mathcal C)\to \calSP_I(\mathcal C)$
is the left adjoint to the embedding of the subcategory of 
$I$-invariant objects with respect to a given interval $I$ in a symmetric monoidal category $\mathcal C$.

\item
$L_\tau\colon \calSP(\mathcal C)\to \calSP_\tau(\mathcal C)$
is the $\tau$-sheafification
for a given Grothendiect topology $\tau$ on $\mathcal C$.\\
$F\simeq_\tau G$ denotes a $\tau$-local equivalence, that means $L_\tau F\simeq L_\tau G$.


\item \label{notation:FtauHXF}
$F_\tau$ is
the sheafification of a given abelian presheaf $F$;
$H_\tau^*(X,F)=H_\tau^*(X,F_\tau)$.

\item \label{notation:hypet}
$\calSP_\et(-)$ are the hypercomplete \'etale sheaves.\\
$L_\et\colon \calSP(-)\to \calSP_\et(-)$ is the left adjoint to the embedding.
A morphisms $v\colon F\to G$ 
is $\et$-equivalence, i.e. $F\simeq_\et G$, 
if 
$L_\et F\simeq L_\et G$. 
Note that 
$\et$-equivalences are called hypercomplete \'etale local equivalences.

\item\label{notation:subscriptbdleqgeq}
$\Spt^*_{\ptaubdleqgeq}(-)=
L_\tau(\Spt^*_{\bdleqgeq}(-))$,
for a given $\tau$,
see Notation \ref{notation:superscriptsubscript} and \ref{notation:SptgeqlSptbSptgeq}.

\item\label{notation:subscriptleq}
$\Spt^*_{\ptauleql{0}}(-)=
L_\tau(\Spt^*_{\leq0}(-))$,
$\Spt^*_{\ptauleq}(-)=
L_\tau(\Spt^*_{\geq}(-))$,
see Notation \ref{notation:superscriptsubscript} and \ref{notation:SptgeqlSptbSptgeq}.



\item\label{notation:subscriptgeq}
$\Spt^*_{\ptaugeql{0}}(-)=
L_\tau(\Spt^*_{\geq0}(-))$,
is the category of $\tau$-locally connective objects;\\
$\Spt^*_{\ptaugeq}(-)=
L_\tau(\Spt^*_{\geq}(-))$,
see Notation \ref{notation:superscriptsubscript} and \ref{notation:SptgeqlSptbSptgeq}.

\item \label{notation:SptgeqlSptbSptgeq}
$\Spt_{\geq b}=(E\in\Spt|\pi_i(E)=0, i<b)$, $\Spt_{\leq b}=(E\in\Spt|\pi_i(E)=0, i>b)$. \\
$\Sptgeq=\bigcup_{l\in\bbZ}\Spt_{\geq b}$,
$\Sptb=\bigcup_{l_0,_1\in\bbZ}(\Spt_{\leq b_1}\cap\Spt_{\geq b_0})$.

\item \label{notation:hetl}
\hetl = hypercomplete \'etale locally. 

\item \label{eq:tetgeqltgeql}
$t_{\geq b}\colon \Spc\to \Spc_{\geq b}$
is the connective cover functor,
\\
$t_{\geq b}\colon \Spc(\Sm_S)\to \Spc_{\geq b}(\Sm_S)$
is the schemewise connective cover functor,
\\
$t^\et_{\geq b}\colon \Spc_\et(\Sm_S)\to \Spc_{\et_{\geq b}}(\Sm_S)$
is the \hetl 
connective cover functor,
that lands in
the subcategory spanned by \hetl $l$-connective objects.
So 
$t^\et_{\geq b}(F)\simeq L_\et t_{\geq b}(R_\et F)$.

\item\label{notation:effveff}$\SptPAniseff(\Sm_k)$ and $\SptPAnisveff(\Sm_k)$ are the subcategories of effective and very effective objects in $\SptPAnis(\Sm_k)$ in sense of \cite{zbMATH06035488}.

\item 
$\mathrm{Map}(-,-)$,
or
$\mathrm{Map}^\Spt_{\calSP}(-,-)$,
denotes the mapping space, 
or the mapping spectrum, for 
an $\infty$-category $\calSP$,
or
a stable $\infty$-category $\calSP$,
respectively.

\item \label{notation:pi(F)(U)}
$\pi_{i}(F)(U)=\pi_i\mathrm{Map}_{\Spt(\Sm_S)}(\Sigma^\infty_{\PP^1}U_+,F)$.
\item \label{notation:piAtau(F)(U)}
$\pi_{i,0}^{\Ao,\tau}(F)(U)=\pi_i\mathrm{Map}_{\Spt^{\PP^{-1}}_{\Ao,\tau}(\Sm_S)}(\Sigma^\infty_{\PP^1}U_+,F),$
$U\in\Sm_S$, $F\in\Spt^{\PP^{-1}}_{\Ao,\tau}(\Sm_S)$.

\item
%
$\Re^{*,\star_0}_{\star_1}\colon \SptS^{*}_{\star_0}(\Sm_k)\to\Spt_{\star_1}(\Et_k)$,\\ 
$\Re^{*,\star_0}_{\star_1}=\Re^{*,\star_1}_{\star_1}L^{*,\star_0}_{\star_1}$.\\
Example: $*=\Gm$, $\star_0=(\A^1,\nis)$, $\star_1=(\et,\hat{l})$.
\item \label{notation:Pconst}
$\Pconst^{\star}_{\star}\colon \Spt_{\star}(\Et_k)\to\Spt_{\star}(\Sm_k)$ 
is the inverse image along $\Et_k\to\Sm_k$;
\\
$\Pconst^{\star_0}_{*,\star_1}=
E^{*,\star_{01}}_{\star_1}\Str^{\star_0}_{*,\star_{01}}\Pconst^{\star_0}_{\star_0}
\colon \Spt_{\star_0}(\Et_k)\to\Spt^{*}_{\star_1}(\Sm_k)$,
where $\star_{01}=\star_0\cup\star_1$.

\item \label{notation:functorsGammas}
$\Gamma^{*,\star_0}_{\star_1}
\colon
\Spt^{*}_{\star_0}(\Sm_k)\to\Spt_{\star_1}(\Et_k)$,
$\Gamma^{*,\star_0}_{\star_1}\simeq \Gamma^{\star_1}_{\star_1} \Fog^{*,\star_0}_{\star_1}$,
\\$\Gamma^{\star_1}_{\star_1}$ is the inverse image along $\Et_k\to\Sm_k$.
%
%


\end{enumerate}

\section{Recollection of conservativities}

\subsection{Completion conservativity}\label{subsection:CompletionConservativity}

For closely related results, see \cite[§2 and §3]{Mathew_2017} and \cite[§2.1]{BachmannRigidity}.
Let $\calSP$ be a prestable presentable symmetric monoidal $\infty$-category;
$n\colon \mathbf{1}\to \mathbf{1}\in\calSP$.
Since $\mathbf{1}$ is strongly dualizable, $\mathbf{1}/l$, see Notation \ref{notation:Fn=cofibFnF}, is such too.
Recall that the subcategory $\calSP_\compln\subset \calSP$ has the right orthogonal the subcategory
$\calSP[n^{-1}]$ spanned by the objects $F$ such that $n\colon F\xrightarrow{\simeq}F$.
\begin{lemma}\label{def:Completion}
The subcategory $\calSP_\compln$ is reflective and
the left adjoint $\calSP\to\calSP_\compln$ 
to the embedding functor is given by $F\mapsto F_\compln$
\begin{equation}\label{eq:Fwn}F_\compln= \colim_\alpha F/l^\alpha\in \calSP.\end{equation}
\end{lemma}
\begin{proof}
The equivalence \eqref{eq:Fwn} implies a cofiber sequence
\begin{equation}\label{eq:sequence:FnFFn}F[n^{-1}]\to F\to F_\compln\end{equation}
where
\begin{equation}\label{eq:Fn=colimFnF}
F[n^{-1}]= \colim(\dots\xrightarrow{n}F\xrightarrow{n} F\xrightarrow{n}\dots)\in \calSP,
\end{equation}
see \Cref{rem:Fninverse}.
Note that $F[n^{-1}]\in \calSP[n^{-1}]$ 
because the arrows on the colimits of the rows in the following commutative diagram are inverse to each other 
\begin{equation*}\xymatrix{
    \dots\ar[r]^{n}\ar[rd]^{1}&F\ar[r]^{n}\ar[rd]^{1}& F\ar[r]^{n}\ar[rd]^{1}&\dots\\
    \dots\ar[r]^{n}&F\ar[r]^{n}\ar[u]^{n}& F\ar[r]^{n}\ar[u]^{n}&\dots.
}\end{equation*}
Then for any $F^\prime\in\calSP_\compln$ the cofiber sequence \eqref{eq:sequence:FnFFn} induces the fibre sequence in $\Spt$
formed by the mapping spectra in the stable category $\calSP$
\begin{equation}\label{eq:FnFFn:fibseqMapSpt}
\calSP(F[n^{-1}], F^\prime)\leftarrow \calSP(F, F^\prime)\leftarrow \calSP(F_\compln, F^\prime).
\end{equation}
and
the left term in \eqref{eq:FnFFn:fibseqMapSpt} is trivial by definition of $\calSP_\compln$. Hence the right side arrow is an equivalence.
Hence \eqref{eq:Fwn} defines the left adjoint to the embedding $\calSP_\compln\to\calSP$.
\end{proof}
\begin{remark}\label{rem:Fninverse}
    Note that $F\mapsto F[n^{-1}]$, see \eqref{eq:Fn=colimFnF}, 
    %
    defines a right adjoint functor to the embedding $\mathcal C[n^{-1}]\to\mathcal C$, 
    though it is not formally used in the text.
\end{remark}
\begin{proposition}\label{prop:complettentoquotient_object}
Let $F\to G\in \calSP$ be a morphism. 
Then 
$F_\compln\simeq G_\compln$,
if and only if 
$F/l\simeq G/l$.
\end{proposition}
\begin{proof}
Suppose $F/l\simeq G/l$.
By induction on $\alpha\in \mathbb Z_{>0}$ if follows that $F/l^\alpha\simeq G/l^\alpha$ for each $\alpha\in\bbZ_{\geq 1}$. 
Hence $F\simeq G$ by \eqref{eq:Fwn}. 
\end{proof}

\begin{example}
Given $l\in\bbZ$, we
define the subcategory 
$\SptHZ_\compln(\mathcal C)$ of 
the $\infty$-category
$\SptHZ(\mathcal C)$
with respect to the morphism $n\colon \HZ\to\HZ\in\SptS$, see Notation \ref{notation:SptZ}.
\end{example}

\begin{proposition}\label{prop:complettentoquotient_category}
Let $\mathcal C$ be a small $\infty$-category,
the functor
$\SptHZ_\compln(\mathcal C)\to \SptHZn(\mathcal C)$
is conservative.
\end{proposition}
\begin{proof}
Since $\cofib(\HZ\xrightarrow{n}\HZ)\simeq\HZn$, where $n\colon \Id_{\SptHZ(\calC)}\to\Id_{\SptHZ(\calC)}$,
the claim follows by \Cref{prop:complettentoquotient_object}.
\end{proof}

\subsection{Whitehead theorem}

\begin{lemma}\label{prop:SWTCnservativityLinearisationSH}
The functor $\Spt_{\geq}\to \SptHZ_{\geq}$ is conservative.
\end{lemma}
\begin{proof}
Since 
the functor of homologies $h^*\colon \SH^\mathrm{top}_{\geq}\to \prod_{l\in\bbZ_{\geq 0}}\mathrm{Ab}$ 
decomposes as $\SH^\mathrm{top}_{\geq}\to \mathbf D^-(\mathrm{Ab})\to \prod_{l\in\bbZ_{\geq 0}}\mathrm{Ab}$
and is conservative
because of the stable Whitehead theorem \cite[Proposition 6.30]{Schwede},
the conservativity of the functor $\SH^\mathrm{top}_{\geq}\to \mathbf D^-(\mathrm{Ab})$ follows.
Finally, the claim follows by the conservativity of the functor
$\Spt\to\SH^\mathrm{top}$, see Notation \ref{notation:SpcSptSHtop},
and because $\mathbf D^-(\mathrm{Ab})$ is the homotopy category of the $\infty$-category $\SptHZ_\geq$, see Notation \ref{notation:SptZ} and \ref{notation:DAb}.
\end{proof}

\begin{theorem}\label{prop:SwnConservativityLinearisation}
Let $\mathcal C$ be a small $\infty$-category. 
The functors 
\begin{gather}
\label{eq:conserv:SHZ}
\SptS_{\geq}(\mathcal C)\to \SptHZ_{\geq}(\mathcal C),\\ 
\label{eq:conserv:SenHZnZ}
\SptS_{\geq,\compln}(\mathcal C)\to \SptHZn_{\geq}(\mathcal C)
\end{gather} are conservative.

\end{theorem}
\begin{proof}
Consider the commutative diagram
\begin{equation}\label{eq:sectioviseSHZ}\xymatrix{
\prod_{X\in \mathcal C}\Spt_{\geq}\ar[r] & \prod_{X\in \mathcal C}\SptHZ_{\geq}\\
\SptS_{\geq}(\mathcal C)\ar[u]\ar[r] & \SptHZ_\geq(\mathcal C)\ar[u]
,}\end{equation}
where the left vertical arrow takes a presheaf of spectra 
$F\in \SptS_\geq(\mathcal C)$ 
to $F(X)$ for each $X$.
Since the left vertical functor in \eqref{eq:sectioviseSHZ} is conservative,
and the upper horizontal functor is conservative by \Cref{prop:SWTCnservativityLinearisationSH},
it follows that the functor \eqref{eq:conserv:SHZ} is conservative. 
Hence \eqref{eq:conserv:SenHZnZ} is conservative because of \Cref{prop:SWTCnservativityLinearisationSH}.
\end{proof}

\section{Framed motives and motivic spectra or spaces.}\label{sect:Framed} 
The framed motives theory 
regarding the Nisnevich localization
was built 
by Garkusha and Panin 
in \cite{garkusha2018framed}
based on \cite{hty-inv,framed-cancel,ConeTheGNP,surj-etale-exc}.
Using some parts of the results of the mentioned articles 
Elmanto, Hoyois, Khan, Sosnilo, Yakerson \cite{elmanto2021motivic} 
reformulated the main results in the theory using the language of $\infty$-categories, providing powerful new tools and perspectives.


In the present article, we
establish the main results from \cite{garkusha2018framed} in the hypercomplete \'etale context \cite{ClausenMathewHypdescetaleKth}.
We use the $\infty$-categorical methods of \cite{elmanto2021motivic} and refer to \cite{bachmann2021remarks} without transferring between the languages
of model categories and $\infty$-categories used in the references mentioned above.


\subsection{Framed motives and $\infty$-loop spaces}\label{sect:FramedNis} 

In this section, 
we briefly
recall the 
main definitions and results 
of the framed motives theory
from 
\cite{garkusha2018framed}
and
\cite{elmanto2021motivic}
respectively.

Firstly, there are 
Voevodsky's framed correspondences, see
\cite{VoevNotesFr},
\cite[Def. 2.1, 2.3]{garkusha2018framed}, 
which form the mapping sets in 
a category $\Fr_+(\Sm_k)$
enriched over pointed sets with objects smooth $k$-schemes.
Recall also
the Nisnevich sheaves of pointed sets
$\Fr(-,Y)=\varinjlim_n\Fr_n(-,Y)$,
see \cite[Def. 2.8]{garkusha2018framed},
for a smooth scheme
$Y$ over $k$. 
The same notation is used 
for the pointed simplicial sheaf $\Fr(-,Y)$
for a simplicial smooth scheme $Y$ over $k$%
.
With respect to
\cite[Def. 5.2]{garkusha2018framed}
we say
that
the framed motive of 
$Y$
is the Nisnevich sheaf
of
$S^1$-spectra
\begin{equation}\label{eq:FrDeltabullet}
(
\Fr(\Delta^\bullet_k\times -,Y),\dots,
\Fr(\Delta^\bullet_k\times -,\Sigma^l_{S^1}Y),
\Fr(\Delta^\bullet_k\times -,\Sigma^{l+1}_{S^1}Y),\dots
)
\end{equation}
defined by
the presheaf of Segal $\Gamma$-spaces
\begin{equation}\label{eq:KmapstoFrDeltabulletdashKtimesY}
K\mapsto \Fr(\Delta^\bullet_k\times -,K\times Y),
\end{equation}
in sense of \cite{zbMATH03446053}.
Since by
\cite[Th. 6.5]{garkusha2018framed}
the $\Gamma$-space 
\eqref{eq:KmapstoFrDeltabulletdashKtimesY}
is special,
we denote 
by
$\Fr(\Delta^\bullet_k\times -,Y)^\gp$
the associated grouplike special $\Gamma$-space.
By 
\cite[Th 10.7]{garkusha2018framed}
it follows 
an equivalence of spectra
\begin{equation}
\label{eq:OmegaGmSigmaPYUsimeqFrDeltaUSigmaSYgp:lochenseesssmschemeU}
(\Omega^\infty_{\PP^1}\Sigma^\infty_{\PP^1} Y)(U)
\simeq 
\Fr(\Delta^\bullet_k\times U,Y)^\gp
\end{equation}
for a 
henselian local essentially smooth scheme $U$ over $k$,
where
\[\Omega^\infty_{\PP^1}\Sigma^\infty_{\PP^1}\colon\Delta^\mathrm{op}\Sm_k\to \SpcPANis(\Sm_k)\to \Spc(\Sm_k).\]

In \cite[§4]{elmanto2021motivic}, the semiadditive $\infty$-category $\Corr^\fr(\Sm_k)$
was introduced, along with a functor
\begin{equation}\label{eq:FrtoCorrfr}
    \Fr_+(\Sm_k)\to\Corr^\fr(\Sm_k),
\end{equation}
and by \cite[Corollary 2.3.27]{elmanto2021motivic} an equivalence in $\Spc(\Sm_k)$
\begin{equation}\label{eq:LZarLAFrYStosimeqLZarLAhfrYS}
L_{\Zar}L_{\Ao}\Fr(-,Y)\xrightarrow{\simeq} L_{\Zar}L_{\Ao}h^\fr(Y),
\end{equation}
where 
$h^\fr(Y)=\Corr^\fr(-,Y)$,
and 
$\Fr(-,Y)$ is called equationally framed correspondences of infinite level
and denoted $h^\mathrm{efr}(Y)$
in \cite[Definition 2.1.2]{elmanto2021motivic}. 
%
%
%
By \cite[Corollary 3.5.16]{elmanto2021motivic}
the composite of the functors 
\begin{equation}\label{eq:functors:SpcSptANisSpcSmk}
\Spc(\Sm_k)\xrightarrow{\Sigma^\infty_{\PP^1}} 
\SpcPANis(\Sm_k)\xrightarrow{\Omega^\infty_{\PP^1}} 
\Spc(\Sm_k).
\end{equation}
satisfies the natural equivalence
\begin{equation}\label{eq:th:Nisnevichfraedmotivicloopsspaces}
\Omega^\infty_{\PP^1}\Sigma^\infty_{\PP^1} F\simeq \calL_\nis \calL_{\Ao} \gamma_*(\gamma^*F^\gp)
\end{equation}
for any $F\in \Spc(\Sm_k)$.

\begin{remark}
    The claims of
    \eqref{eq:OmegaGmSigmaPYUsimeqFrDeltaUSigmaSYgp:lochenseesssmschemeU} and \eqref{eq:th:Nisnevichfraedmotivicloopsspaces} 
    for finite base fields 
    are covered by 
    \cite{DrKyllfinFrpi00} and \cite[Appendix B]{elmanto2021motivic}
    respectively.
    Nonperfect fields $k$ are covered by \cite{SHIfr}.
\end{remark}

\begin{remark}\label{rem:inperfectZpinvert}
    For our purposes, 
    it is enough to use equivalences \eqref{eq:LZarLAFrYStosimeqLZarLAhfrYS} and \eqref{eq:th:Nisnevichfraedmotivicloopsspaces} with $\bbZ[1/p]$-coefficients,
    that have an alternative proof without use of \cite{SHIfr}.
    The claims over $k$ follow from the results over the perfect field $K=k^\mathrm{perf}$ that is 
    the purely inseparable closure of $k$ 
    by the use of 
    the base change along the morphism $e\colon \Spec K\to \Spec k$
    and the application of \Cref{prop:pinsepextSchchk} 
    similarly as in the proof of \Cref{th:SHI:ZPShfrA}. 
    Note that $e_*\calL_\nis\simeq\calL_\nis e_*$ with respect to notation of \Cref{lm:commute:els:Let_AND_:LA_AND_OmegaGm}.
\end{remark}

\subsection{Framed spaces and hypercomplete \'etale localization}\label{sect:Lfret}

In \Cref{sect:HypcEtFrMotives}, we prove the analogues of \eqref{eq:OmegaGmSigmaPYUsimeqFrDeltaUSigmaSYgp:lochenseesssmschemeU} and \eqref{eq:th:Nisnevichfraedmotivicloopsspaces} 
for hypercomplete \'etale framed motives and framed motivic spaces.
In this subsection, we collect some 
more specific results 
used in the text.
The novel part regards the hypercomplete \'etale localization.


Let $S$ be a scheme and consider the adjunction $\gamma^*:\calSP(\Sm_S)\leftadjright\calSP^\fr(\Sm_S):\gamma_*$
for
a presentable $\infty$-category
$\calSP$. 
We denote $\calL^\fr_{\et}$, $\calL^\fr_{\Ao}$ and $\calL^\fr_{\Ao,\et}$ by the localizations that inverts the image of $\gamma^*$ of the morphisms generated by 
hypercomplete \'etale equivalences, $\Ao$-equivalences and both, respectively. Note that these collections of morphisms are strongly saturated by \cite[Lemma 2.10]{bachmann2020norms} since they are closed under sifted colimits and 2-out-of-3. Note further that they are symmetric monoidal by \cite[Proposition 4.1.7.4]{LurieHA} since $X\times(-)$ preserves hypercomplete \'etale sieves and $\Ao$-projections, see \cite[Proposition 3.2.19]{elmanto2021motivic} for the motivic case.

Denote
\[
\essemb\colon\Sm_S\to\EssSm_S,\quad
\essemb^\fr\colon\Corr^\fr_S\to\Corr^\fr(\EssSm_S)
\]
the embedding functors,
and 
denote
$\essemb_\fr^*=(\essemb^\fr)^*$. 
Since the presheaves $h(Y)$ and $h^\fr(Y)$ are continuous
on $\EssSm_S$,
is follows that 
the functor 
$\essemb^*\colon\calSP^*(\Sm_S)\to\calSP^*(\EssSm_S)$
is equivalent to the fully faithful embedding of the
subcategory of continuous presheaves $\calSP^*(\EssSm_S)$,
where $*=\emptyset,\fr$.



\begin{sblemma}\label{sblm:Lhypetfr}
There is an equivalence
\begin{equation}\label{eq:gammausglssimeqglsgammaus}
    \gamma^*_{\locsh}g_*\essemb^*\simeq g^\fr_*\essemb_\fr^*\gamma^*,
\end{equation}
where 
\[g\colon \Smlh_S\to \mathrm{EssSm}_S,\quad  g^\fr\colon \Corr^\fr(\Smlh_S)\to\Corr^\fr(\mathrm{EssSm}_S)\] 
are fully faithful embeddings of the subcategories
spanned by local strictly henselian essentially smooth schemes over $S$,
and
$\gamma_{\locsh}\colon \Smlh_k\to\Corr^\fr(\Smlh_S)$.
A similar claim holds for $\calSP^\gpfr_\Sigma(\Sm_S)$.
\end{sblemma}
\begin{proof}
For any $X,Y\in\Sm_S$, and points $x\in X$, $y\in Y$,
there is a functorial morphism of spaces 
$\Corr^\fr(X,Y)\to \Corr^\fr(X^{sh}_x,Y^{sh}_y)$,
since 
for each $c\in\Corr^\fr(X,Y)$ the support of $c$ is finite over $X$, 
and any connected finite scheme over a local strictly henselian scheme is local strictly henselian.
Hence \eqref{eq:gammausglssimeqglsgammaus} is satisfied on representable objects.
The \eqref{eq:gammausglssimeqglsgammaus} follows because 
$\gamma^*_{\locsh}$, $\gamma^*$,
$\essemb^*$, $\essemb_\fr^*$
preserve colimits being left adjoint, and
$g_*$, $g^\fr_*$ preserve colimits acting schemewise.
\end{proof}

\begin{lemma}\label{lm:Letgammads}
Let $S$ be a scheme.
Let
$\gamma^*\dashv\gamma_*$ denote the adjunctions
\[\calSP(\Sm_S)\leftadjright\calSP^\fr_\Sigma(\Sm_S),\quad \calSP=\Spc,\Spt,\]
see Notation 
\ref{notation:calPfr(calC)}.
There are canonical equivalences
\begin{itemize}
\item[(1)]
$\calL_\et \gamma_*\simeq \gamma_* \calL^\fr_\et$,
\item[(2)]
$\Omega_{S^1} \gamma_*\simeq \gamma_* \Omega^\fr_{S^1}$,
$\Omega_{\Gm} \gamma_*\simeq \gamma_* \Omega^\fr_{\Gm}$,
$\Omega_{\PP^1} \gamma_*\simeq \gamma_* \Omega^\fr_{\PP^1}$,
$\calL_{\Ao} \gamma_*\simeq \gamma_* \calL^\fr_{\Ao}$,
\item[(3)]
$\Sigma^{\infty,\fr}_{S^1}\gamma^*\simeq \gamma^*\Sigma^\infty_{S^1}$,
$\Sigma^{\infty,\fr}_{\Gm}\gamma^*\simeq \gamma^*\Sigma^\infty_{\Gm}$,
$\Sigma^{\infty,\fr}_{\PP^1}\gamma^*\simeq \gamma^*\Sigma^\infty_{\PP^1}$,
$\calL^\fr_{\Ao} \gamma^*\simeq \gamma^* \calL_{\Ao}$,
\item[(4)]
$t_{\geq b} \gamma_*\simeq \gamma_* t_{\geq b}$,
\item[(5)]
$t^\et_{\geq b} \gamma_*\simeq \gamma_* t^\et_{\geq b}$.
\end{itemize}
Similar claims hold for $\calSP^\gpfr_\Sigma(\Sm_S)$.
\end{lemma}
\begin{proof}
(1). 
%
%
Since the schemes $X^{sh}_x$ are the points 
for \'etale topology on $\Sm_S$,
\[g_*\essemb^*\calL_\et\simeq g_*\essemb^*, \quad g^\fr_*\essemb^*_\fr\calL^\fr_\et\simeq g^\fr_*\essemb^*_\fr,\]
where
the right equivalence follows by \Cref{sblm:Lhypetfr}.
Consequently,
\[
g_*\essemb^* \calL_\et\gamma_*\simeq 
g_*\essemb^*\gamma_*\simeq 
\gamma_* g^\fr_*\essemb^*_\fr\simeq 
\gamma_*g^\fr_*\essemb^*_\fr \calL^\fr_\et\simeq
g_*\essemb^*\gamma_* \calL^\fr_\et
.\]
Since the schemes $X^{sh}_x$ form enough set of points,
$g_*$ is conservative with respect to $\et$-equivalences.
Then since both $\gamma_*\calL^\fr_\et F$ and $\calL_\et \gamma_*F$ are 
hypercomplete \'etale sheaves over any $F\in\Spcfr(\Sm_S)$,
the claim follows.
%
(2). 
The claim on $\Omega_{S^1}$ 
holds because the functor acts schemewise.
Recall that 
$\gamma_*$ preserves limits and sifted colimits by \cite[Proposition 3.2.15]{elmanto2021motivic} on the level of spaces and 
all small colimits on the level spectra by \cite[3.5.2]{elmanto2021motivic}.
Then the claims on $\Omega_{\Gm}$ and $\calL_{\Ao}$ follow 
because $\gamma$ commutes with $-\times_S\Delta^n_{S}$ and $-\times\bbG_m^{\times n}$, $-\times\PP^{\times n}$, $l\in\bbZ$,
see \cite[page 33, Proposition 3.2.14]{elmanto2021motivic} for $\calL_{\Ao}$.
(3)
follows from (2),
and since $\gamma^*$ preserves
morphisms $\Ao\times X\to X$,
$\gamma^*$ preserves $\Ao$-equivalences,
i.e.
$\calL^\fr_{\Ao} \gamma^*\simeq \gamma^* \calL_{\Ao}$.
(4) 
follows from the definition because $\gamma$ is equivalence on objects. 
(6) follows by (1) and (4). 
\end{proof}

\begin{lemma}\label{lm:gammauslsAeteq}
The functors $\gamma_*$ in \Cref{lm:Letgammads}
are conservative with respect to 
equivalences, $\Ao$-equivalences, and $\et$-equivalences.
The functors $\gamma^*$ 
preserve 
$\Ao$-equivalences, and $\et$-equivalences.
\end{lemma}
\begin{proof}
The claims on $\gamma_*$
follow
because the functor $\gamma$ induces the isomorphism on objects.
The first claim on $\gamma^*$
follows because $\gamma(h(\Ao\times_k X)\to h(X))=\gamma(h(\Ao\times_k X))\to \gamma(h(X))$.
The second claim on $\gamma^*$
follows because $\gamma_*$ preserves $\et$-local objects, i.e., \'etale hypersheaves, by \Cref{lm:Letgammads}(2).
\end{proof}


\begin{remark}
    The functor $\gamma$ decomposes as 
    \[\Sm_k\xrightarrow{\gamma^{\Fr_+}}\Fr_+(\Sm_k)\to\Corr^\fr(\Sm_k).\]
    Claims of \Cref{lm:Letgammads,lm:gammauslsAeteq} 
    hold for $\gamma^{\Fr_+}$ as well by the same arguments as for $\gamma$.
\end{remark}



\section{Algebraic and purely inseparable extensions of base fields}\label{sect:PurelyinsepextensionsASigma}

     For a field extension $K/k$,
denote 
\begin{equation*}\begin{array}{lcll}
e_{\Sch}&\colon&\Sch_k\to \Sch_K; &X\mapsto X\times_k\Spec K,\\
e&\colon&\Sm_k\to \Sm_K; &X\mapsto X\times_k\Spec K,
\end{array}
\end{equation*}
and we use the same notation for $\Corr^\fr(-)$.
Then there are the adjunctions
\[
e^*_\Sch\dashv e^\Sch_*,\quad 
e^*\dashv e_*
\]
with respect to
$\calSP^*_\star(-)$,
where
$\calSP^*_\star(-)$ is $\calSP(-)$ or $\calSP^\fr_{\Sigma}(-)$
for a presentable
$\infty$-category
$\calSP$.


\begin{lemma}\label{lm:commute:els:Let_AND_:LA_AND_OmegaGm}
    For any algebraic field extension $K/k$,\begin{itemize}
    \item[(1)] $\calL_{\cet}e^\Sch_*\simeq e^\Sch_*\calL_\cet$, $\calL_{\cet}e_*\simeq e_*\calL_\cet$, 
    \item[(2)] $\calL_{\Ao}e^\Sch_*\simeq e^\Sch_*\calL_{\Ao}$, $\calL_{\Ao}e_*\simeq e_*\calL_{\Ao}$,\\
               $\Omega_{\Gm}e^\Sch_*\simeq e^\Sch_*\Omega_{\Gm}$, $\Omega_{\Gm}e_*\simeq e_*\Omega_{\Gm}$,
    \end{itemize}
\end{lemma}
\begin{proof}
    The claim (1) 
    follows because
    the functors $e$ and $e_\Sch$
    preserve 
    \'etale coverings and
    strictly henselian essentially smooth local schemes.

    The claim (2) 
    follows because 
    the functors $e$ and $e_\Sch$
    commute with the endofunctors $-\times_{(-)}\Delta^n_{(-)}$, 
    and $-\times\bbG_m^{\times n}$, $l\in\bbZ$,
    while the functor $e_*$ preserves colimits and limits because it acts schemewise.
\end{proof}

\begin{lemma}\label{lm:pdgreeextensionlinfting}
Let $K=k(\alpha)$ be a prime extension defined by a polynomial $p\in k[t]$ of degree $d$.
(1)
There is a natural splitting
\begin{equation}\label{eq:spliitingprimeKkSchF}
e_*^\Sch e_\Sch^*(F)\simeq F^\prime\oplus F,\quad F\in\Spt^\HZgpfr_{\Sigma,\Ao}(\Sch_k)[1/d_\varepsilon],
\end{equation}
with respect to the adjunction
\[e^*_\Sch\colon\Spc^\HZgpfr_{\Sigma,\Ao}(\Sch_k)\leftadjright\Spc^\HZgpfr_{\Sigma,\Ao}(\Sch_K)\colon e_*^\Sch,\]
see Notation \ref{notation:SptHZgpfr}.

(2)
Suppose in addition that $K/k$ is purely inseparable.
Then \eqref{eq:spliitingprimeKkSchF}
is an equivalence.

\end{lemma}
\begin{proof}
(1)
%
Consider the element $f\in h^\fr(\Spec K)(\Spec k)$,
given by 
the
span
\[\xymatrix{
&Z(p)\ar[ld]\ar[rd]^{\simeq}&\\
\Spec k&&\Spec K
}
\]
and
trivialisation of $N_{Z(p)/\Ao_k}$
defined by $p$.
Let $X\in \Sch_k$, 
$    u\colon X\times_k\Spec K\to X $
be the projection.
The element defined by $f$
in
\[\pi_0\Map(\hZgpfr(X),\hZgpfr(X\times_k\Spec K))\cong \pi_0(\hZgpfr(X\times_k\Spec K)(X))\]
is denoted in what follows by $f$ as well.
So we have a morphism 
$f\colon X\to X\times_k\Spec K$
in the  homotopy category
of $\Spc^\Zgpfr_{\Sigma,\Ao}(\Sch_k)$.
By 
\cite[Proposition B.1.4]{elmanto2021motivic}
there is the equality of the classes  
\begin{equation}\label{eq:ucircfdeps}
u\circ f=d_\varepsilon\in \pi_0\Map(\hZgpfr(X),\hZgpfr(X\times_k\Spec K)).
\end{equation}
Hence $f$ induces splitting \eqref{eq:spliitingprimeKkSchF}. 



(2)
Applying \eqref{eq:spliitingprimeKkSchF} to $F=\hZgpfr(X_K)$, we get
    \begin{equation}\label{eq:splittingXKX}
        \hZgpfr(X_K)\simeq \hZgpfr(X)^\prime\oplus \hZgpfr(X). 
    \end{equation}
    Since \[X_{K\times K}=X_K\times_k\Spec K=X_K\times_k\Spec K[t]/(p)\cong X_K\times_k\Spec K[s]/(s^p)\simeq_{\Ao}X_K,\]
    the splitting \eqref{eq:splittingXKX}
    goes to an equivalence along the base change from $k$ to $K$.
    So $\hZgpfr(X_K)^\prime\simeq 0$. On the other hand, by \eqref{eq:spliitingprimeKkSchF} applied
    to $F=\hZgpfr(X)^\prime$, we get
    \begin{equation*}
        \hZgpfr(X_K)^\prime\simeq \hZgpfr(X)^{\prime\prime}\oplus \hZgpfr(X)^\prime 
    \end{equation*}
    Hence $\hZgpfr(X)^\prime\simeq 0$.
\end{proof}
%


Define 
$p_\varepsilon=\sum_{i=0}^{p-1}\langle(-1)^{i}\rangle\in\ZF_1(\mathrm{pt}_k,\mathrm{pt}_k)$,
where 
$\langle\lambda\rangle$ for $\lambda\in k^\times$,
is given by
\[(0,\Ao_k,\lambda t,\Ao_k\to\Spec k)\in\Fr_1(\mathrm{pt}_k,\mathrm{pt}_k).\]

\begin{lemma}\label{lm:addp_invertingofpimpliesinveringofpesilon}
For any field $k$ of exponential characteristic $p$, 
there is the canonical isomorphism
of rings
$R[1/p_\varepsilon]\simeq R[1/p]$,
where \[R = \ovZF_*(\mathrm{pt}_k,\mathrm{pt}_k).\]
\end{lemma}
\begin{proof}
Let $p=2$, then $p_\varepsilon=2_\varepsilon=2=p$. 

Let $p$ be odd.
Then 
the class of $p_\varepsilon$ in 
$R/(2_\varepsilon)$
equals $1$, and
\[(R/(2_\varepsilon))[1/p, 1/p_\varepsilon]\simeq (R/(2_\varepsilon))[1/p].\]
Since 
$(d_1 d_2)_\varepsilon=(d_1)_\varepsilon (d_2)_\varepsilon$, for any $d_1,d_2\in\bbZ_{\geq 0}$,
we have the equality
$p_\varepsilon 2_\varepsilon =  p 2_\varepsilon$ in $R$.
Hence
\[R[1/2_\varepsilon][1/p, 1/p_\varepsilon]\simeq R[1/2_\varepsilon][1/p].\]
Thus
$R[1/d,1/d_\varepsilon]\simeq R[1/d].$
Conversely,
\[R[1/p_\varepsilon,1/p]\simeq R[1/p_\varepsilon]\]
because
$p_\varepsilon p_\varepsilon=p p_\varepsilon$.
%
%
%
\end{proof}


\begin{proposition}\label{prop:pinsepextSchchk}
Let 
$K/k$
be a purely inseparable extension of  exponential characteristic $p$.
Denote by $p^\prime$ any one of 
$p$ or $p_{\varepsilon}$.

(1)
The adjunction 
\begin{equation}\label{eq:adj:PShfrASchinvchk}
\PShfrgpAS(\Sch_k)[1/p^\prime]\leftadjright \PShfrgpAS(\Sch_K)[1/p^\prime], 
\end{equation}
is an equivalence.

(2)
The adjunction
\begin{equation}\label{eq:adj:PShfrASminvchk}
\PShfrgpAS(\Sm_k)[1/p^\prime]\leftadjright \PShfrgpAS(\Sm_K)[1/p^\prime], 
\end{equation}
is
coreflective.
The functor 
\begin{equation}\label{eq:PShfrgpSSmKSmk}
e_*\colon\PShfrgpS(\Sm_K)[1/p^\prime]\to\PShfrgpS(\Sm_k)[1/p^\prime]
\end{equation}
is essentially surjective on $\Ao$-invariant objects.

The same results hold for 
$\Ab^\fr_{\Sigma}(-)$.
\end{proposition}
\begin{proof}
    The claims for $p^\prime=p$ follow 
    from the claims for $p^\prime=p_{\varepsilon}$
    by \Cref{lm:addp_invertingofpimpliesinveringofpesilon}.
    Let $p^\prime=p_{\varepsilon}$.

Without loss of generality we may assume that $K/k$ is finite because $e^\Sch_*$
preserves limits, and moreover, we may assume that $K/k$ is prime.
    The equivalence \eqref{eq:adj:PShfrASchinvchk} follows by \Cref{lm:pdgreeextensionlinfting}(2) because of \Cref{prop:SwnConservativityLinearisation}.
Consider the coreflection
\begin{equation}\label{eq:restr:PShfrSchkSmk}
s^*\colon\PShfrgp(\Sm_k)\leftadjright\PShfrgp(\Sch_k)\colon s_*, 
\end{equation}
induced by the fully faithful functor $\Sm_k\to\Sch_k$,
and similar one for $K$.
The coreflection \eqref{eq:adj:PShfrASminvchk}
follows 
from the equivalence \eqref{eq:adj:PShfrASchinvchk}
because of the coreflection induced by $s^*\dashv s_*$ on $\PShfrgpA(-)$.
Moreover, 
since $\calL_{\Ao}s_*\simeq\calL_{\Ao}s_*$,
the functor $s_*$ is essentially surjective on $\Ao$-invariant objects. 
Hence \eqref{eq:PShfrgpSSmKSmk}
is essentially surjective on $\Ao$-invariant objects
because of \Cref{lm:commute:els:Let_AND_:LA_AND_OmegaGm}(2), 
and the equivalence \eqref{eq:adj:PShfrASchinvchk}. 
%
\end{proof}





\section{Rigidity and strict $\Ao$-invariance for framed abelian presheaves.}
\label{sect:RigiditySHIZFrperfk}



\begin{theorem}[\protect{\cite{FramedRigidityLoc}}]
\label{prop:RigidityXhx:ZfrnalphaA:etlocconst}
Let $k$ be a perfect field.
Let $n\in \mathbb Z$ such that $(n,\echark k)=1$. 
Let 
$F\in\ZnPShASigma(\Fr_+(\Sm_k))$ be quasi-stable, or
$F\in \ZPShfrASigma(\Sm_k)$
such that
$F\simeq \varinjlim F_{n^\alpha}$.
Then 
for any $X\in \Sm_k$, and $x\in X$, 
there is a canonical isomorphism
\begin{equation}\label{eq:FXhxsimeqFX}F(X^h_x)\simeq F(x).\end{equation}
Consequently,
any $F\in \SptShatnfrAS(\Sm_k)$ 
is locally constant with respect to the \'etale topology on $\Sm_k$.
\end{theorem}
\begin{proof}
    Since $k$ is perfect,
    for each point $x\in X$, $x$ is an essentially smooth scheme over $k$.
    Hence by \cite[Theorem 7.10]{FramedRigidityLoc}
    \eqref{eq:FXhxsimeqFX} holds for 
    quasi-stable $F\in\ZnPSh_\Sigma(\Fr_+(\Sm_k))$, and consequently, for 
    $F\in \ZnPShfrASigma(\Sm_k)$,
    that implies \eqref{eq:FXhxsimeqFX}.
    
    For any $F\in \SptShatnfrAS(\Sm_k)$, for each $l\in\bbZ$, 
    the presheaf $\pi_l(F)$ is in $\ZPShfrASigma(\Sm_k)$ 
    and $\pi_l(F)\simeq \varinjlim_{\alpha}\pi_l(F)_{n^\alpha}$
    because objects $h^\fr(U)\in \SptSfrSigma(\Sm_k)$ for $U\in\Sm_k$ are compact.
    So the second claim follows because for each closed point $x\in X$, there is a retraction $X^h_x\to x$ since $k$ is perfect, see \cite[Lemma 3.10]{RigiditySmAff}.
\end{proof}

\begin{remark}\label{rem:RigidityFr}
More generally, 
\eqref{eq:FXhxsimeqFX} holds for 
quasi-stable $F\in\ZnPShASigma(\Fr_+(\Sm_k))$ 
such that $F\simeq\varinjlim_{\alpha}F_{n^\alpha}$.
\end{remark}

\begin{theorem}
\label{th:SHI:ZPShfrA}
For any 
$F\in \ZPShfr_{\Ao,\Sigma}(\Sm_k)$ 
such that $F\simeq\varinjlim_{\alpha}F_{n^\alpha}$,
there is a canonical isomorphism
\[H_\cet^*(X\times\Ao,F_{\cet})\simeq H^*_\cet(X,F_{\cet}).\]
\end{theorem}
\begin{proof}
The claim for a perfect field $k$
follows from
the strict homotopy invariance theorem for locally constant torsion presheaves
\protect{\cite[Lemma 9.23]{MVW}, \cite[page 126, XV 2.2]{SGAfour}}
because of 
\Cref{prop:RigidityXhx:ZfrnalphaA:etlocconst}.

Let $k$ be any field, and $K=k^\mathrm{perf}$ be purely inseparable closure.
By \Cref{prop:pinsepextSchchk}(2) $F=e_*(F^\prime)$ for some $F\in\ZPShfr_{\Ao,\Sigma}(\Sm_k)$. Then by the above the presheaves $H_\cet^*(X\times\Ao,F^\prime_{\cet})\in\ZPShfr_{\Sigma}(\Sm_K)$ are $\Ao$-invariant. Hence by \Cref{lm:commute:els:Let_AND_:LA_AND_OmegaGm} the presheaves $H_\cet^*(X\times\Ao,F_{\cet})\in\ZPShfr_{\Sigma}(\Sm_k)$ are $\Ao$-invariant.
\end{proof}

\section{\'Etale local framed motives and motivic spectra}\label{sect:HypcEtFrMotives}

In this section, we build analogs of \cite{garkusha2018framed} and \cite{elmanto2021motivic}
for hypercomplete \'etale $l$-complete localization.
The central results are 
concentrated in \Cref{sect:OmegaPAetRNisetcompln}
and prove the formula for the composte functor 
\[
\Omega_{\Gm}^\infty\Sigma_{\PP^1}^\infty
\colon 
\Delta^\mathrm{op}\Sm_k\to\SptPhatnAhypet(\Sm_k)\to\Spt(\Sm_k)
\] 
like in \Cref{sect:FramedNis}.
\Cref{sect:LAet-SptSfrAet} computes the functor
\[\calL^{\gpfr}_{\Ao,\et}=E^{\gpfr}_{\Ao,\et}L^{\gpfr}_{\Ao,\et}\colon 
\SptSfrSigma(\Sm_k)\to\SptSfrAet(\Sm_k)\to\SptSfrSigma(\Sm_k)
.\]
\Cref{sect:complnhypetA1frCancellation} computes the functor
\[\Omega_{\Gm}^\infty\Sigma_{\Gm}^\infty\colon 
\SptShatnAhypet(\Sm_k)\to\SptPhatnAhypet(\Sm_k)\to\SptShatnAhypet(\Sm_k).\]
\Cref{sect:SCancelation} computes the functor
\[\Omega_{S^1}^\infty\Sigma_{S^1}^\infty\colon 
\Spcfr(\Sm_k)\to\SptfrShatnAhypet(\Sm_k)\to\Spcfr(\Sm_k).\]
\Cref{sect:OmegaPAetRNisetcompln}
summarises three previous steps.
\Cref{sect:Appl} discusses the applications.


\subsection{$S^1$-stable framed hypercomplete motivic localization}\label{sect:LAet-SptSfrAet}
Let $k$ be a field.

\begin{proposition}\label{prop:Let(_A)}
    If $F\in\SptShatnfrSigma(\Sm_k)$ is $\Ao$-invariant,
    then $\calL_{\et}(\gamma_*F)$ is $\Ao$-invariant.
\end{proposition}
\begin{proof}
    Since $F\simeq \varprojlim_{l} t_{\leq b}F$,
    where $t_{\leq b}F\in \SptS^\fr_{\Sigmaleq{l}}(\Sm_k)\simeq (\SptS_{\leq b})^\fr_{\Sigma}(\Sm_k)$,
    because of \Cref{lm:varprojlimlLettleqlFleftarrowLetFtoLetvarprojlimltleqlF},
    without loss of generality $F\in \SptS^\fr_{\Sigmaleq}(\Sm_k)$, and consequently $\gamma_*F\in\SptS_{\Sigmaleq}(\Sm_k)$.  
    Then $\calL_\et \gamma_*F\in \SptS_{\etleq}(\Sm_k)$ 
    is $\et$-locally bounded above
    because $\calL_\et$ is t-exact \cite[Example 2.2]{ClausenMathewHypdescetaleKth} and \cite[Remark 1.3.2.8]{LurieHA}.
    
    The \'etale sheaves associated with 
    $\pi_{q}(\calL_\et \gamma_*F)$ equal
    the ones associated with 
    $\pi_{q}(\gamma_*F)$,
    see Notation \ref{notation:pi(F)(U)}.
    Hence 
    for any 
    smooth scheme $U$ over $k$,
    there are the spectral sequences
    \begin{equation}\label{eq:secseq:HApipiLetA_HpipiLet}\begin{array}{lcl}
    H_{\cet}^{p}(\Ao\times U,\pi_{q}(\gamma_*F)_{\cet}) &\Rightarrow& \pi_{l}((\calL_\et \gamma_*F)^{\Ao})(U),\\
    H_{\cet}^{p}(U,\pi_{q}(\gamma_*F)_{\cet}) &\Rightarrow& \pi_{l}(\calL_\et \gamma_*F)(U)
    ,\end{array}\end{equation}
    see Notation \ref{notation:FtauHXF} for the left side, where $l=-p+q$, 
    that conditionally converge by \Cref{prop:specseq_for_etalepisheaves:perfect:boundedncomplete}. 
    Since $F$ is $\Ao$-invariant, 
    the presheaves $\pi_{q}(\gamma_*F)$ are $\Ao$-invariant,
    see Notation \ref{notation:pi(F)(U)}.
    Hence by \Cref{th:SHI:ZPShfrA} the presheaves
    $H_{\cet}^{p}(-,\pi_{q}(\gamma_*F)_{\cet})$ are $\Ao$-invariant.
    So the projection $\Ao\times X\to X$ induces the isomorphisms 
    \[
    H_{\cet}^{p}(\Ao\times U,\pi_{q}(\gamma_*F)_{\cet}) \xleftarrow{\simeq}
    H_{\cet}^{p}(U,\pi_{q}(\gamma_*F)_{\cet}),
    \]
    and because of the spectral sequences 
    \eqref{eq:secseq:HApipiLetA_HpipiLet}
    it follows that 
    \begin{equation}\label{eq:pilLetgammaFA1UcongpilLetgammaFU}
    \pi_{l}((\calL_\et\gamma_*F)^{\Ao})(U)\xleftarrow{\simeq}
    \pi_{l}(\calL_\et\gamma_*F)(U)        
    \end{equation}
    for all $l\in\bbZ$ and $U$ as above.
    By \eqref{eq:pilLetgammaFA1UcongpilLetgammaFU} the presheaves $\pi_l(\calL_\et\gamma_*F)$ in $\Ab(\Sm_k)$ are $\Ao$-invariant.
    Hence $\calL_\et\gamma_*F$ is $\Ao$-invariant.
\end{proof}

\begin{theorem}\label{th:LA1et}
    For any $F\in\SptShatnfrSigma(\Sm_k)$, 
    there is a natural equivalence 
    \begin{equation}\label{eq:LfrAetFsimeqLfretLfeAF}
        \calL^\fr_{\Ao,\et}(F)\simeq \calL^\fr_{\et}\calL^\fr_{\Ao}(F).
    \end{equation}
\end{theorem}
\begin{proof}
    Since 
    $\gamma_*\calL_{\et}^\fr(F)\simeq \calL_\et\gamma_*(F)$
    by \Cref{lm:Letgammads}, 
    the claim follows by \Cref{prop:Let(_A)}.
\end{proof}


\begin{remark}\label{rem:LetLAFrY}
    By the similar arguments to \Cref{prop:Let(_A)}
    and \Cref{th:LA1et},
    it follows that
    $\calL_{\et}(\gamma_{\Fr_{+}}^*\gamma^{\Fr_{+}}_*F)\in\SptSet(\Sm_k)$ is $\Ao$-invariant
    for a quasi-stable $\Ao$-invariant $F\in\SptS(\Fr_+(\Sm_k))_\compln$,
    see Notation \ref{notation:gammaFr}.
    Consequently,
    for any 
    quasi-stable 
    $F\in\SptS(\Fr_+(\Sm_k))_\compln$,
    there is a natural equivalence 
    \[\calL_{\Ao,\et}(\gamma_{\Fr_{+}}^*\gamma^{\Fr_+}_*F)\simeq \calL_{\et}\calL_{\Ao}(\gamma_{\Fr_{+}}^*\gamma^{\Fr_+}_*F).\]
\end{remark}

\subsection{$S^1$-Cancellation}\label{sect:SCancelation}



\begin{theorem}[$S^1$-cancellation]\label{th:SCancellation}
Let $S$ be a scheme.
    
    (0)
    The adjunctions 
    \[
    \Sigma^{\infty,\gpfr,*,\star}_{S^1}\colon
    \Spc^{\gpfr,*}_{\star}(\Sm_S)\leftadjright \Spt^{\gpfr,*}_\star(\Sm_S)\colon
    \Omega^{\infty,\gpfr,*,\star}_{S^1}
    ,\]
    see Notation \ref{notation:supersubscriptcombPPGm},
    for 
    any $*$ and $\star$ as in Notation \ref{notation:Psstardash_lists}
    are
    coreflections. 

    (1) 
    The functor
    $\Sigma^{\infty,\gpfr}_{S^1}$
    (for $*=\emptyset$ and $\star=\emptyset$)
    (1a)
    commutes with 
    $\calL_{\Ao}^\gpfr$, 
    $\Sigma_{\Gm}^\gpfr$,
    $\Omega_{\Gm}^\gpfr$,
    (1b)
    preserves 
    $\et$-equivalences,
    and 
    (1c)
    is conservative 
    with respect to 
    $\et$-equivalences,
    see Notation \ref{notation:hypet}.
\end{theorem}
\begin{proof}
(Claim 0) 
The claim follows
because of the equivalences
\[
\Spt^{\gpfr,*}_\star(\Sm_S)=
(\operatorname{Stab}\Spc)^{\gpfr,*}_\star(\Sm_S)\simeq
\operatorname{Stab}(\Spc^{\gpfr,*}_\star(\Sm_S))
,\]
and
the coreflection
\[
\Spc^{\gpfr,*}_\star(\Sm_S)
\leftadjright
\operatorname{Stab}(\Spc^{\gpfr,*}_\star(\Sm_S))
,\]
provided by
the categories $\Spt^{\gpfr,*}_\star(\Sm_S)$
are additive and \cite[Corollary 2.1.5]{zbMATH07614298}.

We proceed with
(Claim 1).\\
(Step 1)
    Since the functors
    $\infSigmaSfrgp$, $\infOmegaSfrgp$
    act schemewise 
    and commute with filtered colimits
    they preserve
    $\Ao$-invariant objects and
    hypercomplete \'etale equivalences.
    This
    confirms (1b).
\\
(Step 2)
    Since the functors
    $\infSigmaSfrgp$, $\infOmegaSfrgp$
    are 
    fully faithful,
    they detect
    $\Ao$-invariant objects and
    are conservative with respect to
    hypercomplete \'etale equivalences.
    This
    confirms (1c).
\\ 
(1a for $\calL_{\Ao}$)
%
    $\infSigmaSfrgp$
    preserves $\Ao$-equivalences 
    since $\Ao$-equivalences are generated by the images of the morphisms $X\times\Ao\to X$, $X\in\Sm_S$.
    Hence by (Step 1)
    $\infSigmaSfrgp$
    commutes with $\calL_{\Ao}$.
\\    
(1a for $\Sigma_{\Gm}$)
        $\Sigma^{1,\gpfr}_{S^1}$
        commutes with $\Sigma_{\Gm}$,
        because smash-products commute with smash-products.
    Hence
    $\infSigmaSfrgp$
    commutes with $\Sigma_{\Gm}$.
\\
(1a for $\Omega_{\Gm}$)
        $\Sigma^{1,\gpfr}_{S^1}$ commutes with 
        $\Omega_{\Gm}$,
        because it is a schemewise smash-product, 
        while
        smash-products commute with fibers.
    Hence
    $\infSigmaSfrgp$
    commutes with $\Omega_{\Gm}$.
\end{proof}

\begin{lemma}[\protect{\cite[Lemma 15.1]{druzhinin2023trivial}}]
\label{lm:LreptLprepcpreserveexact}
Given 
a presentable $\infty$-category $\mathcal C$ with reflective subcategories 
$\mathcal C_{\alpha}$, 
$\mathcal C_{\beta}$,
$\mathcal C_{\alpha\beta}=\mathcal C_{\alpha}\cap\mathcal C_{\beta}$.
The following conditions are equivalent:
\begin{itemize}
\item[(1)] 
The canonical morphism $\Lrep_{\beta}\Lrep_{\alpha}\to\Lrep_{\beta\alpha}$ is an isomorphism.
\item[(2)] $\calL_{\beta}$ preserves $\mathcal C_{\alpha}$.
\item[(3)] $\calL_{\alpha}$ preserves $\beta$-equivalences.
\end{itemize}
Here
$\Lrep_{\alpha}$ denotes the corresponding localization endofunctor on 
$\mathcal C$ taking values in $\mathcal C_{\alpha}$, and
$\alpha$-equivalence in $\mathcal C$ is a morphism that maps to an equivalence under $\calL_{\alpha}$,
and similarly for $\beta$.
\end{lemma}

\begin{corollary}[Hypercomplete \'etale motivic localization]\label{cor:LA1et}
    Let $k$ be a field.
    For any $F\in\Spcfrgp_{\Sigma,\compln}(\Sm_k)$, 
    there is a natural equivalence 
    \begin{equation}\label{eq:LgpfrAetFsimeqLgpfretLgpfeAF}
        \calL^{\gpfr}_{\Ao,\et}(F)\simeq \calL^{\gpfr}_{\et}\calL^\gpfr_{\Ao}(F).
    \end{equation}
\end{corollary}
\begin{proof}
    The claim is equivalent to the fact that
    $\calL^{\fr}_{\Ao}$ in \eqref{eq:LfrAetFsimeqLfretLfeAF}
    preserves 
    h. \'et. l. equivalences 
    by \Cref{lm:LreptLprepcpreserveexact}.
    So the claim follows from
    \Cref{th:LA1et},
    because 
    $\Sigma^{\infty,\gpfr}_{S^1}$
    commutes with $\calL^\fr_{\Ao}$
    and
    is exact and conservative with respect to  
    h. \'et. l. equivalences
    by
    \Cref{th:SCancellation}(1).
\end{proof}

\subsection{Motivic cancellation}\label{sect:complnhypetA1frCancellation}


Let $S$ be a scheme. 

\begin{lemma}\label{lm:commute:MonoidalStructure_fr}
    The functor 
    \[\gamma^*\colon\SptS_{\et}(\Sm_S)\to \SptSfr_{\et}(\Sm_S)\]
    preserves the symmetric monoidal structure $\wedge$,
    in particular, 
    \[\gamma^*({\mathbf 1}_k)\simeq {\mathbf 1}^\fr_k, \quad \gamma^*(F)\wedge \gamma^*(E)\simeq \gamma^*(F\wedge E).\]
    The same applies to the functor
    $\gamma^*\colon\SptHZn_{\Ao,\et}(\Sm_S)\to \SptHZnfr_{\Ao,\et}(\Sm_S)$.
\end{lemma}
\begin{proof}
    The symmetric monoidal structures $\wedge$ on the $\infty$-categories $\SptS_{\et}(\Sm_S)$ and $\SptSfr_{\et}(\Sm_S)$ are
    both given by the Day convolution structure induced by $\times_S$ on $\Sm_S$
    via left Kan extension along the functors
    $M\colon \Sm_S\to \SptS_{\et}(\Sm_S)$ and $M^\fr\colon \Sm_S\to \SptS_{\et}(\Sm_S)$,
    respectively.
    So the claim follows because of the equivalence $\gamma^* M\simeq M^\fr$.
    The claim on $\SptHZn_{\et}(\Sm_S)$ and $\SptHZnfr_{\et}(\Sm_S)$ follows similarly.
\end{proof}

Recall that the equivalence $\PP^1\simeq \bbG_m^{\wedge 1}\wedge S^1$, induces 
the equivalence 
\begin{equation}\label{eq:SptPfrAetSptGfrAetSmS}\SptPfrAet(\Sm_S)\simeq\SptGfrAet(\Sm_S),\end{equation} 
and recall the equivalence
\begin{equation}\label{eq:SptSfrAetSptSgpfrAetSmS}\SptS^\fr(\Sm_S)\simeq\SptS^\frgp(\Sm_S),\end{equation}
provided by the $\infty$-category $\Spt$ is stable.


\begin{theorem}
\label{th:GmCancellation:SfrAet} 
    (a)
    The adjunction of the $\infty$-categories
    \begin{equation}\label{eq:Cancellation:adj:SptSwefrAet}\Sigma_\Sphere\colon 
    \SptSfr_{\Ao,\et,\compln}(\Sm_S)\leftadjright \SptSfr_{\Ao,\et,\compln}(\Sm_S)
    \colon \Omega_\Sphere\end{equation}
    is an equivalence,
    where $\Sphere$ denotes $\bbG_m^{\wedge 1}$ or $\PP^{\wedge 1}$.
    Consequently, the adjunction of the $\infty$-categories
    \[\Sigma^\infty_\Sphere\colon 
    \SptSfr_{\Ao,\et,\compln}(\Sm_S)\rightadjleft \Spt^{\Sphere,\fr}_{\Aet,\compln}(\Sm_S) 
    \colon \Omega^\infty_\Sphere\]
    is an equivalence.

    (b)
    The adjunctions of the $\infty$-categories 
    \[
    \SpcfrgphatnAet(\Sm_S)\stackrel{\simeq}{\leftadjright} \SpcGfrgphatnAet(\Sm_S)
    \leftadjright
    \SptGfrhatnAet(\Sm_S)
    \stackrel{\simeq}{\leftadjright} \SptPhatnfrAet(\Sm_S)
    \]
    are coreflections, and 
    the left and the right ones are equivalences.
\end{theorem}
\begin{proof}
    (a)
    Since     
    $
    \SptShatnAhypet(\Sm_S) 
    \simeq 
    \Spt^{\Sphere,\fr}_{\Aet,\compln}(\Sm_S) 
    $, 
    by \cite[Theorem 3.1]{bachmann2021remarks},
    $\Sphere\in\SptSAet(\Sm_S)$ is invertible.
    Hence by \Cref{lm:commute:MonoidalStructure_fr} the motive of $\Sphere$ in $\SptSfrAet(\Sm_S)$ is invertible.
    So the equivalence \eqref{eq:Cancellation:adj:SptSwefrAet} follows.
    \par(b)
    The first equivalence follows from point (a) 
    and \Cref{th:SCancellation}(1).
    The second adjunction is a reflection by \Cref{th:SCancellation}(0).
    The right adjunction is an equivalence by \eqref{eq:SptPfrAetSptGfrAetSmS}.
\end{proof}
\begin{remark}
    A similar proof to point (a) above
    was shown by Bachmann to the second author, however, it was independently discovered by the first author.
\end{remark}

Inspired by the formulation in \cite[Theorem 1.1 and the sentence just before]{bachmann2023notes}, the hyper\'etale motivic infinite loop space theory can be expressed by the next result:

\begin{corollary}
[Hypercomplete \'etale recognition theorem]
Let $k$ be a field.
There is an equivalence of $\infty$-categories
\[
\Spc^{\gpfr}_{\Ao,\et}(\Sm_k)\simeq\SptPAetveff(\Sm_k),
\]
where the right side is the subcategory of very effective spectra \cite{zbMATH06035488}.
\end{corollary}
\begin{proof}
%
    We prove the claim after $l$-completion, 
    using the commutative diagram of 
    adjunctions of $\infty$-categories
    \[
    \xymatrix{
    \SpcAet(\Sm_k)\ar@<0.5ex>[r]\ar@<0.5ex>[d]^{\simeq}&
    \SptPhatnAetveff(\Sm_k)\ar@<0.5ex>[r]\ar@<0.5ex>[l]\ar@<0.5ex>[d]^{\simeq}& 
    \SptPhatnAhypet(\Sm_k)\ar@<0.5ex>[l]\ar@<0.5ex>[d]^{\simeq}
    \\
    \SpcAet(\Sm_k)\ar@<0.5ex>[r]\ar@<0.5ex>[u]&
    \SpcfrgphatnAet(\Sm_k)\ar@<0.5ex>[r]\ar@<0.5ex>[l]\ar@<0.5ex>[u]& 
    \SptPfrgphatnAet(\Sm_k)\ar@<0.5ex>[l]\ar@<0.5ex>[u]
    .
    }
    \]
    Here
    the right side adjunction in 
    the first row is
    a coreflection by definition
    and 
    the right side adjunction in 
    the second row is
    a coreflection by 
    \Cref{th:GmCancellation:SfrAet}(b). 
    The right side vertical adjunction is an equivalence because the 
    functor
    $\SptPfrgpAnis(\Sm_k)\to \SptPANis(\Sm_k)$
    is an equivalence by the reconstruction equivalence \cite[Theorem 3.5.12]{elmanto2021motivic}
    and because it preserves and detects h. \'et. stable motivic equivalences
    since
    $\Spt^{\gpfr}(\Sm_k)\to \SptS(\Sm_k)$
    preserves and detects h. \'et. local equivalences
    by definition.
    Since the left vertical adjunction is an equivalence and the functors from the left to the right in the left side square are essentially surjective,
    the 
    left vertical adjunction is an equivalence.

    The claim with $\mathbb Q$-coefficients follows from 
    the recognition theorem in the Nisnevich setting
    \cite[Theorem 3.5.14]{elmanto2021motivic}
    in view of \Cref{rem:Fr1/p}. 
    Since the $\infty$-categories $\Spc^{\gpfr}_{\Ao,\et,\hat{p}}(\Sm_k)$ and $\Spt^{\PP^{-1}}_{\veff,\sAo,\et,\hat{p}}(\Sm_k)$
    are trivial by \cite[Appendix A]{bachmann2021remarks}, see Notation \ref{notation:kcahark=p} for $p$, the claim with $\bbZ$-coefficients follows.
\end{proof}

\subsection{Motivic $\infty$-loop spaces and homotopy groups}\label{sect:OmegaPAetRNisetcompln}

Let $k$ be a field.

\begin{theorem}
\label{lm:etalefraedmotivicloopsspaces}
Let $\Sphere$ denote $\PP^{\wedge 1}$ or $\bbG_m^{\wedge 1}$.

(a)
The composite of the functors 
\[
\calS^{\gpfr}_\Sigma(\Sm_k)\xrightarrow{\Sigma^{\infty,\gpfr}_{\Sphere,\Aet}} 
\calS^{\Sphere,\gpfr}_{\Aet}(\Sm_k)\xrightarrow{\Omega^{\infty,\gpfr}_{\Sphere,\Aet}} 
\calS^{\gpfr}_\Sigma(\Sm_k),
\]
see Notation \ref{notation:supersubscriptPPGm:short},
where 
$\calS$ is $\Spc$ or $\Spt$, 
satisfies the following natural equivalence
\begin{equation}\label{eq:OmegaSigmaGm_simeq_LetLAgp}
\Omega^{\infty,\gpfr}_{\Sphere,\Aet}\Sigma^{\infty,\gpfr}_{\Sphere,\Aet} F\simeq \calL_\et^\gpfr \calL_{\Ao}^\gpfr F
\end{equation}
for any $F\in \Spc^{\gpfr}_\Sigma(\Sm_k)_\compln$.

(b)
The composite of the functors 
\begin{equation}\label{eq:SpckSptPetAkSpck}
\SptS(\Sm_k)\xrightarrow{\Sigma^\infty_{\Sphere,\Aet,\compln}} 
\SptS^{\Sphere}_{\Aet,\compln}(\Sm_k)\xrightarrow{\Omega^\infty_{\Sphere,\Aet,\compln}}
\SptS(\Sm_k),
\end{equation}
satisfies the following natural equivalence
\[\label{eq:cor:etalefraedmotivesformula}
\Omega^\infty_{\Sphere,\Aet,\compln}\Sigma^\infty_{\Sphere,\Aet,\compln} F\simeq \calL_\et \calL_{\Ao} \gamma_*(\gamma^*(F)^\gp_\compln)
\]
for any $F\in \Spc(\Sm_k)$.

\end{theorem}
\begin{proof}
    (a) 
    The claim for $\calS=\Spc$
    summarises 
    \Cref{th:LA1et} and \Cref{th:GmCancellation:SfrAet}(a), and 
    \Cref{cor:LA1et} and \Cref{th:GmCancellation:SfrAet}(b), 
    respectively.
    The claim for $\calS=\Spt$ follows because the functors $\Spc\to\Spt$ induces an equivalence on \eqref{eq:OmegaSigmaGm_simeq_LetLAgp}.

    (b) 
The claim follows from (a) 
because of the diagram
\begin{equation}\label{eq:diagramSpckSptPetAkSpck}\xymatrix{
\SptS^{\gpfr}_{\Sigma,\compln}(\Sm_k)\ar[r]^{\Sigma^\infty_{\Sphere,\Aet}}& 
\SptSpherefrgphatnAet(\Sm_k)\ar[r]^{\Omega^\infty_{\Sphere,\Aet}}\ar@<1ex>[d]^{\gamma^\gp_*}\ar[rd]|{\text{Lm \ref{lm:Letgammads}(2)}}&
\SptS^{\gpfr}_{\Sigma,\compln}(\Sm_k)\ar@<1ex>[d]^{\gamma^\gp_*}
\\
\SptS_\compln(\Sm_k)\ar[ru]|{\text{Lm \ref{lm:Letgammads}(3)}}\ar[r]^{\Sigma^\infty_{\Sphere,\Aet}}\ar@<1ex>[u]^{\gamma_\gp^*}& 
\SptSpherehatnAhypet(\Sm_k)\ar[r]^{\Omega^\infty_{\Sphere,\Aet}}\ar@<1ex>[u]^{\gamma_\gp^*}&
\SptS_\compln(\Sm_k)
,}\end{equation}
where the middle vertical functors are equivalences
because by
    \cite[Theorem 3.5.12]{elmanto2021motivic}, we have
    \[\SptSphereAet(\Sm_k)\simeq\SptSpherefrAet(\Sm_k)\simeq\SptSpherefrgpAet(\Sm_k).\]
\end{proof}

\begin{theorem}\label{cor:etalefraedmotivesformula:Fr}
(a)
Consider the composite functor
\[
\Omega^\infty_{\PP^1}\Sigma^\infty_{\PP^1}\colon
\Delta^\mathrm{op}\Sm_k\to \SpcAet(\Sm_k)\xrightarrow{\Sigma^\infty_{\PP^1}} \SpcPetA(\Sm_k)\xrightarrow{\Omega^\infty_{\PP^1}} \SpcAet(\Sm_k)\to \Spc(\Sm_k),
\]
see
Notation 
\ref{notation:supersubscriptPPGm:short}.
For any simplicial smooth scheme $Y$ over $k$ and 
for any strictly henselian local essentially smooth $U$ over $k$,
there is an equivalence in $\SptS$
\begin{equation}\label{eq:intro:OmegaGmSigmaPYUcomplnsimeqFrDeltaUSigmaSYgpcompln:strictlochenseesssmschemeU}
(\Omega^\infty_{\PP^1}\Sigma^\infty_{\PP^1}\Sigma^\infty_{S^1} Y)(U)_\compln
\simeq 
(\Fr(\Delta^\bullet_k\times U,\Sigma^\infty_{S^1}Y)^\gp)_\compln,
\end{equation}
where subscript $\compln$ means $l$-completion, see Notation \ref{notation:calPcompln},
and $(\Fr(\Delta^\bullet_k\times U,\Sigma^\infty_{S^1}Y)^\gp)_\compln$ 
is an $\Omega_{S^1}$-spectrum of pointed simplicial sets,
see 
\eqref{eq:FrDeltabullet},
and consequently, there is an equivalence of spaces
\begin{equation}\label{eq:OmegaGmSigmaPYUcomplnsimeqFrDeltaUSigmaSYgpcompln:strictlochenseesssmschemeU}
(\Omega^\infty_{\PP^1}\Sigma^\infty_{\PP^1} Y)(U)_\compln
\xleftarrow{\simeq} 
(\Fr(\Delta^\bullet_k\times U,Y)^\gp)_\compln,\end{equation}
where the left side is the $l$-completion of the group completion of the special Segal's $\Gamma$-space \eqref{eq:KmapstoFrDeltabulletdashKtimesY}. 

(b) Let $\Sphere$ denote $\PP^{\wedge 1}$ or $\bbG_m^{\wedge 1}$.
Consider the composite functor
\[\Omega^\infty_{\Sphere}\Sigma^\infty_{\Sphere}\Sigma^\infty_{S^1}\colon
\Delta^\mathrm{op}\Sm_k\to \SpcAet(\Sm_k)\xrightarrow{\Sigma^\infty_{\Sphere}\Sigma^\infty_{S^1}} \SptShereetA(\Sm_k)\xrightarrow{\Omega^\infty_{\Sphere}} 
\SptSetA(\Sm_k)\to \SptS(\Sm_k),\]
see Notation \ref{notation:supersubscriptPPGm:short}.
For any simplicial smooth scheme $Y$ over $k$ and 
strictly henselian local essentially smooth $U$ over $k$,
there is an equivalence in $\SptS$
\begin{equation}\label{eq:OmegaGmSigmaPSigmaSYUcomplnsimeqFrDeltaUSigmaSYgpcompln:strictlochenseesssmschemeU}
(\Omega^\infty_{\Sphere}\Sigma^\infty_{\Sphere}\Sigma^\infty_{S^1} Y)(U)_\compln
\xleftarrow{\simeq} 
(\Fr(\Delta^\bullet_k\times U,\Sigma^\infty_{S^1}Y)^\gp)_\compln,\end{equation}
where the right side is given by 
an
$\Omega_{S^1}$-spectrum of pointed simplicial sets
in 
the sense of
\cite{zbMATH01698557}.

(c)
There are the natural \'etale local isomorphisms in $\Ab(\Sm_k)$ 
\[\begin{array}{lclclclr}
\pi_{i,0}^{\Ao,\et}(\Sigma^\infty_{\PP^1} Y)_l&\simeq& \pi_i(\Fr(\Delta^\bullet_k\times -,Y))_l,
&\quad&
\pi_{i,0}^{\Ao,\et}(\Sigma^\infty_{\PP^1} Y)_{/l}&\simeq& \pi_i(\Fr(\Delta^\bullet_k\times -,Y))_{/l},
\\
\pi_{0,0}^{\Ao,\et}(\Sigma^\infty_{\PP^1} Y)_l&\simeq& \ZF(\Delta^\bullet_k\times -,Y)_l,
&\quad&
\pi_{0,0}^{\Ao,\et}(\Sigma^\infty_{\PP^1} Y)_{/l}&\simeq& \ZF(\Delta^\bullet_k\times -,Y)_{/l},
\end{array}\]
for each $i\in\bbZ_{> 0}$,
and vanishings for $i\in\bbZ_{<0}$.

\end{theorem}

\begin{proof}
Firstly, we prove (b).
Since
$\Fr(\Delta^\bullet_k\times U,\Sigma^\infty_{S^1}Y)^\gp$
is an $\Omega_{S^1}$-spectrum, see \Cref{sect:FramedNis},
the right side of \eqref{eq:OmegaGmSigmaPSigmaSYUcomplnsimeqFrDeltaUSigmaSYgpcompln:strictlochenseesssmschemeU}
is an $\Omega_{S^1}$-spectrum.
So it is enough to prove 
an equivalence 
\begin{equation*}\label{eq:OmegaGmSigmaPYUcomplnsimeqFrDeltaUSigmaSYcompln:strictlochenseesssmschemeU}
(\Omega^\infty_{\Sphere}\Sigma^\infty_{\Sphere}\Sigma^\infty_{S^1} Y)(U)_\compln
\xleftarrow{\simeq} 
(\Fr(\Delta^\bullet_k\times U,\Sigma^\infty_{S^1}Y))_\compln.
\end{equation*}
in $\Spt$.
The claim follows from
\Cref{lm:etalefraedmotivicloopsspaces}(b), 
because
by \cite[Lemma 2.3.27]{elmanto2021motivic}, see
\eqref{eq:LZarLAFrYStosimeqLZarLAhfrYS},
\begin{equation*}\label{eq:LetLAFrYStosimeqLZarLAhfrYS}
\begin{array}{lcl}
\calL_{\et}\calL_{\Ao}\gamma_*h^\fr(Y)&\simeq& \calL_{\et}\calL_{\Ao}\Fr(-,Y),
\\
\calL_{\et}\calL_{\Ao}\gamma_*(\gamma^*(\Sigma^\infty_{S^1}h(Y)))
&\simeq& 
(\calL_{\et}\calL_{\Ao}\Fr(-,\Sigma^\infty_{S^1}Y))
.\end{array}\end{equation*}

Point (a) follows from (b) for $\Sphere=\PP^{\wedge 1}$.
Point (c) follows from (b) as well.
\end{proof}


\begin{corollary}\label{th:connectivity}
    \par (a) 
    The functors
    \begin{equation}\label{eq:SigmainftySPconnectivity}
    \SptS_{\et,\compln}(\Sm_k)
    \xrightarrow{L^{\et,\compln}_{\Ao}}
    \SptShatnAhypet(\Sm_k)
    \xrightarrow{\Sigma^\infty_\Gm}
    \SptSGhatnAhypet(\Sm_k)
    \end{equation}
    preserve 
    \hetl connective objects.
%
 %
%
    \par (b) 
    There are equivalences of $\infty$-categories
    \begin{equation}\label{eq:Rigidity:connective:SptPAetSmkSptSAetSmkSptSetEtk}
    \xymatrix{
    \Spt^{\bbG_m^{-1}}_{\Ao,\et\geq0,\compln}(\Sm_k)
    \ar[r]^\simeq
    &
    \SptS_{\Ao,\et\geq0,\compln}(\Sm_k)
    \ar[r]^\simeq\ar[d]^\simeq&
    \SptS_{\et\geq0,\compln}(\Et_k)\ar[d]^\simeq
    \\
    &
    \CMon^\gp_{\Ao,\et,\compln}(\Sm_k)
    \ar[r]^\simeq&
    \CMon^\gp_{\et,\compln}(\Et_k)
    }
    \end{equation}
    where the first row consists of the subcategories of \hetl connective objects, see Notation \ref{notation:subscriptgeq} and \ref{notation:subscriptintersect}. 
%
%
%
%
\end{corollary}
\begin{proof}
    (a)
    By 
    \eqref{eq:OmegaGmSigmaPSigmaSYUcomplnsimeqFrDeltaUSigmaSYgpcompln:strictlochenseesssmschemeU}
    any very effective 
    $F\in\SptSGAet(\Sm_k)$ 
    is \hetl connective.
    Hence
    the composite functor $\Sigma^\infty_{\Gm}L^{\et,\compln}_{\Ao}$
    in \eqref{eq:SigmainftySPconnectivity}
    preserves 
    \hetl connective objects.
    Recall the sequence of equivalences
    \begin{equation}\label{eq:Rigidity:SptPAetSmkSptSAetSmkSptSetEtk}
    \SptSGhatnAhypet(\Sm_k)
    \stackrel{\Omega^\infty_\Gm\vdash\Sigma^\infty_\Gm}{\simeq}
    \SptShatnAhypet(\Sm_k)\simeq\SptShatnhypet(\Et_k)
    \end{equation}
    provided by \cite[Theorem 3.1]{bachmann2021remarks}.
    Since the adjunction
    $\Omega^\infty_\Gm\vdash\Sigma^\infty_\Gm$
    is an equivalence, 
    and $\Omega^\infty_\Gm$ preserves \hetl connective objects,
    it follows that 
    $L^{\et,\compln}_{\Ao}$
    in \eqref{eq:SigmainftySPconnectivity}
    preserves 
    \hetl connective objects.
    Recall the 
    adjunction 
    \begin{equation}\label{eq:SptSAetSmSptSetEt}
    L^{\et,\compln}_{\Ao}
    \colon\SptS_{\et,\compln}(\Sm_k)
    \leftadjright\SptShatnAhypet(\Sm_k)\colon
    E^{\A^1,\et,\compln}_{\et,\compln}.
    \end{equation}
    Since \eqref{eq:SptSAetSmSptSetEt} is a reflection,
    it follows that
    \begin{equation}\label{eq:SigmasimeqSigmaLcircE}
    \Sigma^\infty_{\Gm}
    \simeq
    (\Sigma^\infty_{\Gm}L^{\et,\compln}_{\Ao})
    \circ
    (E^{\A^1,\et,\compln}_{\et,\compln}).
    \end{equation}
    As proven above the functor 
    $\Sigma^\infty_{\Gm}L^{\et,\compln}_{\Ao}$ preserves \hetl connective objects.
    The functor
    $E^{\A^1,\et,\compln}_{\et,\compln}$
    preserves \hetl connective objects
    by definition.
    Hence by \eqref{eq:SigmasimeqSigmaLcircE}
    $\Sigma^\infty_{\Gm}$
    preserves \hetl connective objects.

    
    (b)
    Consider the right side equivalence in \eqref{eq:Rigidity:SptPAetSmkSptSAetSmkSptSetEtk} given by the adjunction
    $\Xi^{\et,\hat{l}}_{\A^1,\et,\hat{l}}\dashv\Gamma^{\A^1,\et,\hat{l}}_{\et,\hat{l}}$, see Notation \ref{notation:Pconst}, \eqref{notation:functorsGammas}.
    Since the inverse image functor $\Spt_{\et}(\Et_k)\to\Spt_{\et}(\Sm_k)$ preserves \hetl connective objects, and $L^{\et,\compln}_{\Ao}$ does this 
    as proven above, 
    it follows that
    $\Xi^{\et,\hat{l}}_{\A^1,\et,\hat{l}}$ does this too. 
    Since 
    the functor $E^{\A^1,\et,\compln}_{\et,\compln}$, see \eqref{eq:SptSAetSmSptSetEt},   
    preserves
    \hetl connective objects,
    it follows that $\Gamma^{\A^1,\et,\hat{l}}_{\et,\hat{l}}$ does this.
    So the right side upper equivalence in \eqref{eq:Rigidity:connective:SptPAetSmkSptSAetSmkSptSetEtk} follows.

    Consider the left side equivalence in \eqref{eq:Rigidity:SptPAetSmkSptSAetSmkSptSetEtk} given by the adjuncton
    $\Sigma^\infty_{\Gm}\dashv\Omega^\infty_{\Gm}$. As proven above $\Sigma^\infty_{\Gm}$ preserves \hetl connective objects. Since $\Omega^\infty_{\Gm}$ preserves \hetl connective objects by definition, the left side upper equivalence in \eqref{eq:Rigidity:connective:SptPAetSmkSptSAetSmkSptSetEtk} follows.

    %
    %
    %
    %
%

    The $\infty$-category $\CMon^\gp_{\Ao,\et,\compln}(\Sm_k)$
    is equivalent to the subcategory of $\SptS_{\Ao,\et,\compln}(\Sm_k)$
    generated by the objects of the form $\Sigma^\infty_{S^1}X$, $X\in\Sm_k$, via colimits.
    The latter subcategory is equivalent to 
    $\SptS_{\Ao,\et\geq,\compln}(\Sm_k)$, 
    because the functors $\SptS_{\et,\compln}(\Sm_k)\leftadjright\SptShatnAhypet(\Sm_k)$
    preserve \hetl connective objects by the above.
    Thus 
    the middle vertical equivalence in \eqref{eq:Rigidity:connective:SptPAetSmkSptSAetSmkSptSetEtk}
    holds, 
    and the right side vertical equivalence in \eqref{eq:Rigidity:connective:SptPAetSmkSptSAetSmkSptSetEtk}
    follows similarly.
\end{proof}

\subsection{Comparison results and realization functors}
\label{sect:Appl}
Let $k$ be a field.

\begin{theorem}\label{cor:nisconget}
\par (a) Let $E\in\SpcPhatnAnis(\Sm_k)$ be effective 
and denote by the same symbol its image in $\SpcPhatnAhypet(\Sm_k)$. 
The natural homomorphism of presheaves 
\begin{equation}\label{eq:piA1nissimeqpiA1et}
\pi_{i+j,j}^{\Ao,\nis}(E)
\rightarrow 
\pi_{i+j,j}^{\Ao,\et}(E),
\end{equation}
is an \'etale local isomorphism,
for each $i\in\bbZ$, $j\in\bbZ_{\geq 0}$.
\par (b) 
The following
diagram of 
$\infty$-categories 
is commutative, see Notation \ref{notation:effveff},
\begin{equation}\label{eq:diag:cor:nisconget}
\xymatrix{
\SptSGhatnAniseff(\Sm_k)\ar[d]^{\Omega^\infty_\Gm}
\ar[r]
&
\SptSGhatnAeteff(\Sm_k)
\ar[d]^{\Omega^\infty_\Gm}
\\
\SptS_\nis(\Sm_k)\ar[r]^{L^\nis_\et}
&
\SptS_\et(\Sm_k)
}
\end{equation}
\end{theorem}
\begin{proof}
    Both claims follow because of
    \eqref{eq:OmegaGmSigmaPYUcomplnsimeqFrDeltaUSigmaSYgpcompln:strictlochenseesssmschemeU}
    and
    \eqref{eq:OmegaGmSigmaPYUsimeqFrDeltaUSigmaSYgp:lochenseesssmschemeU}.
\end{proof}

Recall the equivalence 
\begin{equation}\label{eq:SpcPAetSptGAetSmS}\SpcPAet(\Sm_S)\simeq\SptGAet(\Sm_S),\end{equation} 
induced by the equivalence $\PP^1\simeq \bbG_m^{\wedge 1}\wedge S^1$.
The following definition generalizes 
\cite[Def. 2.4.6]{haine2023spectral}.
\begin{definition}\label{def:Renisetcompln}
Let $*\in 
\{\PP^{-1},\bbG_m^{-1},S^{-1},\gp\}$.
Using 
\eqref{eq:SpcPAetSptGAetSmS}
\eqref{eq:Rigidity:SptPAcetSmkSptScetEtk}, 
\eqref{eq:intro:Rigidity:SptPAetSmkSptSAetSmkSptSetEtk},
and 
\eqref{eq:Rigidity:connective:SptPAetSmkSptSAetSmkSptSetEtk},
we define the functors
\begin{equation}\label{eq:Rhypethypet}
\Re^{*,\A^1,\hypet,\compln}_{\hypet,\compln}=
(\Pconst^{\hypet,\compln}_{*,\A^1,\hypet,\compln})^{-1}
\end{equation}
of the form
\[
\begin{array}{lcl}
  \Spc^{\PP^{-1}}_{\Ao,\et,\compln}(\Sm_k)&\to & \Spc_{\et,\compln}(\Et_k),\\
  \SptS^{\bbG_m^{-1}}_{\Ao,\et,\compln}(\Sm_k)&\to & \SptS_{\et,\compln}(\Et_k),\\
  \SptS_{\Ao,\et,\compln}(\Sm_k)&\to & \SptS_{\et,\compln}(\Et_k),\\
  \CMon^{\gp}_{\Ao,\et,\compln}(\Sm_k)&\to & \CMon^{\gp}_{\et,\compln}(\Et_k)
\end{array}
\]
respectively.
Then
we 
define 
\[\Re^{*,\A^1,\Nis}_{\hypet,\compln}=\Re^{*,\A^1,\hypet,\compln}_{\hypet,\compln}L^{*,\A^1,\nis}_{\et,\compln}
\]
see \eqref{eq:intro:def_RNishypetcompln}, \eqref{eq:intro:def_RNishypetcompln_CMongp}.
%
%
\end{definition}

The following result extends \cite[Theorem 7.1]{BachmannRigidity} covering grouplike $\infty$-commutative monoids.
\begin{theorem}\label{th:Remonoidadjoint}
Let
$*\in \{\PP^{-1},\bbG_m^{-1},S^{-1},\gp\}$,
and
$\star\in \{\nis,(\et,\compln)\}$.
The functor 
$\Re^{*,\A^1,\star}_{\et,\compln} 
$
from \Cref{def:Renisetcompln}
is 
symmetric monoidal
and 
left adjoint 
to the functor 
$\Pconst^{\et,\compln}_{*,\A^1,\star}$.
\end{theorem}
\begin{proof}
Equivalence \eqref{eq:Rhypethypet}
implies 
that 
$\Re^{*,\A^1,\hypet,\compln}_{\hypet,\compln}$ is symmetric monoidal
and there is an adjunction 
\begin{equation}\label{adj:RehypethatnhypethatndashvPconsthypethatnhypethatn}
\Re^{*,\A^1,\hypet,\compln}_{\hypet,\compln}\dashv (\Pconst^{\et,\compln}_{*,\A^1,\et,\compln})    
.\end{equation}
Since
$L^{*,\A^1,\nis}_{\et,\compln}$
and
$\Re^{*,\A^1,\hypet,\compln}_{\hypet,\compln}$
are symmetric monoidal,
$\Re^{*,\A^1,\Nis}_{\hypet,\compln}$ is symmetric monoidal too.
Finally,
the composite of the adjunctions
$
L^{*,\A^1,\nis}_{\et,\compln}\dashv E_{*,\A^1,\nis}^{\et,\compln}
$
and
\eqref{adj:RehypethatnhypethatndashvPconsthypethatnhypethatn}
is
the required adjunction
$\Re^{*,\A^1,\Nis}_{\hypet,\compln}
\dashv
\Pconst^{\et,\compln}_{*,\A^1,\Nis}$.
\end{proof}

\begin{theorem}\label{th:piAniswnBetenhypetANDLoopspectraandspacesANDRealisationfunctors}
Let $k$ be a field, and $l\in\bbZ$ be invertible in $k$.

(a)
    %
    For any $E\in \SptSGAniseff(\Sm_k)$,
    there is
    a natural equivalence in $\SptShatnhypet(\Et_k)$
    \[\Gamma^{\motpref,\nis}_{\et,\compln}(E)\simeq 
    \Re^{\motpref,\Nis}_{\hypet,\compln}(E),\]
    where
    $\Gamma^{\motpref,\nis}_{\et,\compln}=L^{\et}_{\et,\compln}\Gamma^{\motpref,\nis}_{\et}$, see Notation \ref{notation:functorsGammas}.
(b)
    The following diagram of $\infty$-categories is commutative 
    \begin{equation*}\label{eq:RealisLoopFunctorsCommute}\xymatrix{
    \SpcPAniseff(\Sm_k)\ar[r]^{L^{\Pmotpref,\nis}_{\et,\compln}}\ar[d]^{\Omega^\infty_{\PP^1}}&
    \SpcPhatnAhypet(\Sm_k)\ar[r]^{\Re^{\Pmotpref,\et,\compln}_{\et,\compln}}\ar[d]^{\Omega^\infty_{\PP^1}}&
    \Spt_{\et,\compln}(\Et_k)\ar[d]^{\Omega^\infty_{S^1}}
    \\
    \CMon^\gp_{\Ao,\nis}(\Sm_k)\ar[r]^{L^{\unsmotpref,\nis}_{\et,\compln}}\ar@{-->}[d]|{k=\bbC}&
    \CMon^\gp_{\Ao,\et,\compln}(\Sm_k)\ar[r]^{\Re^{\unsmotpref,\et,\compln}_{\et,\compln}}\ar@{-->}[d]|{k=\bbC}&
    \CMon^\gp_{\et,\compln}(\Et_k)\ar@{-->}[d]|{k=\bbC}
    \\
    %
    \Spc_{\Ao,\nis}(\Sm_\bbC)\ar[r]^{L^{\unsmotpref,\nis}_{\cet}}&
    \Spc_{\Ao,\cet}(\Sm_\bbC)\ar[r]^{Be}&
    \Spc
    ,}\end{equation*}
    where 
    the upper three rows are considered for any field $k$,
    the bottom row is considered for $k=\bbC$, and $Be$ denotes the Betti realization.

\end{theorem}
\begin{proof}
    (a)
    %
    It follows from \Cref{cor:nisconget}(b)
    that there is an equivalence
    $\Gamma^{\motpref,\nis}_{\hypet,\compln} 
    \cong 
    \Gamma^{\motpref,\hypet}_{\hypet,\compln} L^{\motpref,\nis}_{\hypet}
    $. 
    Then
    \[\begin{array}{lcl}
    \Gamma^{\motpref,\nis}_{\hypet,\compln}(E) 
    &\stackrel{\text{Th. \ref{cor:nisconget}(b)}}{\simeq} &
    L^\et_{\et,\compln}\Gamma^{\motpref,\hypet}_\hypet L^{\motpref,\nis}_\hypet(E)\\
    &\stackrel{}{\simeq}& 
    \Gamma^{\motpref,\hypet,\compln}_{\hypet,\compln} L^{\motpref,\nis}_{\hypet,\compln}(E)\\
    &\stackrel{\eqref{eq:Rigidity:SptPAcetSmkSptScetEtk}}{\simeq}& 
    (\Pconst^{\hypet,\compln}_{\motpref,\hypet,\compln})^{-1}L^{\motpref,\nis}_{\hypet,\compln}(E)\\
    &\stackrel{\eqref{eq:Rhypethypet}}{=}&
    \Re^{\motpref,\hypet,\compln}_{\hypet,\compln}L^{\motpref,\nis}_{\hypet,\compln}(E)\\
    &\stackrel{\text{Def. \ref{def:Renisetcompln}}}{=}&
    \Re^{\motpref,\nis,\compln}_{\hypet,\compln}(E).
    \end{array}\]

    (b)
    Denote by
    $\operatorname{Im}_{\SptSAnis(\Sm_k)}(\Omega^\infty_\Gm)$
    the full subcategory in
    $\SptSAnis(\Sm_k)$
    that is an essential image of the functor
    $\Omega^\infty_{\Gm}\colon\SptSGAniseff(\Sm_k)\rightarrow\SptSAnis(\Sm_k).$
    Then the vertical functors between the upper two rows in \eqref{eq:RealisLoopFunctorsCommute} are equivalent to the composite functors 
    in the diagram
    \begin{equation}\label{eq:ShvconnectSptPdsquare}
    \xymatrix{
    \SptSGAniseff(\Sm_k)\ar[r]^{L^{\motpref,\nis}_{\et,\compln}}\ar[d]^{\Omega^\infty_\Gm}&
    \SptSGhatnAhypet(\Sm_k)\ar[r]\ar[d]^{\Omega^\infty_\Gm}_{\simeq}&
    \Spt_{\et,\hat{l}}(\Et_k)\ar[d]_{\simeq}
    \\
    \operatorname{Im}_{\SptSAnis(\Sm_k)}(\Omega^\infty_\Gm)\ar[r]^{L^{\nis}_{\et,\compln}}\ar[d]^{t_{\nis\geq0}}&
    \SptShatnAhypet(\Sm_k)\ar[r]\ar[d]^{t_{\et\geq0}}&
    \SptShatnhypet(\Et_k)\ar[d]^{t_{\et\geq0}}
    \\
    \SptS_{\A^1,\nis\geq0}(\Sm_k)\ar[r]^{L^{\nis}_{\et,\compln}}\ar[d]^{\simeq}&
    \SptS_{\A^1,\et\geq0,\compln}(\Sm_k)\ar[r]\ar[d]^{\simeq}&
    \Spt_{\et\geq 0,\compln}(\Et_k)\ar[d]^{\simeq}
    \\
    \CMon^\gp_{\Ao,\Nis}(\Sm_k)\ar[r]^{L^{\nis}_{\et,\compln}}&
    \CMon^\gp_{\Ao,\et,\compln}(\Sm_k)\ar[r]&
    \CMon^\gp_{\et,\compln}(\Et_k)&
    .}\end{equation}
    Here
    $t_{\nis\geq0}$, $t_{\et\geq0}$,
    $L^{\nis}_{\et,\compln}$
    denote the restrictions of the respective functors
    to the subcategories in the diagram.
    The squares in the left two columns are well defined and commutative by 
    \Cref{cor:nisconget}(b) and
    \Cref{th:connectivity}(b),
    see
    \eqref{eq:diag:cor:nisconget},
    \eqref{eq:Rigidity:connective:SptPAetSmkSptSAetSmkSptSetEtk},
    respectively.
    The commutativity of the upper right square is provided by 
    %
    \cite[Theorem 3.1]{bachmann2021remarks}, see \eqref{eq:Rigidity:SptPAcetSmkSptScetEtk},
    and \Cref{th:connectivity}(b), see \eqref{eq:Rigidity:connective:SptPAetSmkSptSAetSmkSptSetEtk}.
    The commutativity of the other squares 
    holds because
    the direct image along the embedding $\Et_k\to\Sm_k$
    commutes with 
    $t_{\et\geq 0}$ on $\SptS_{\et}(-)$,
    and there is
    the equivalence $\SptS_{\et\geq 0}(-)\simeq\CMon^\gp_\et(-)$.

    Let $k=\bbC$,
    then $\CMon^\gp_{\et,\compln}(\Et_k)\simeq \CMon^\gp_\compln$,
    and 
    there are 
    commutative squares
    \begin{equation}\label{eq:diag:squares:CMobgpAetcomplncetSmkEtk}
    \xymatrix{
    \CMon^\gp_{\Ao,\et,\compln}(\Sm_k)\ar[r]^>>>>>{\Re^{\motpref,\et,\compln}_{\et,\compln}}_{\simeq}&\CMon^\gp_\compln\\
    \CMon^\gp_{\Ao,\cet,\compln}(\Sm_k)\ar[r]^>>>>>{Be}\ar[u]^{L^{\Ao,\cet,\compln}_{\et}}&\CMon^\gp_\compln\ar[u]^{\simeq}
    }
    \quad
    \xymatrix{
    \CMon^\gp_{\Ao,\et,\compln}(\Sm_k)\ar[r]^>>>>>{\Re^{\et,\compln}_{\et,\compln}}_{\simeq}\ar[d]^{E_{\Ao,\et,\compln}^{\cet}}&\CMon^\gp_\compln\ar[d]^{\simeq}\\
    \CMon^\gp_{\Ao,\cet,\compln}(\Sm_k)\ar[r]^>>>>>{Be}&\CMon^\gp_\compln
    .}
    \end{equation}
    Here the commutativity of the left square 
    holds because
    $Be$ is left adjoint to the functor
    $\CMon^\gp_{\compln}\to\CMon^\gp_{\Ao,\cet,\compln}(\Sm_k)$,
    and
    $\Re^{\motpref,\et,\compln}_{\et,\compln}$ is left adjoint to the functor
    $\CMon^\gp_{\compln}\to\CMon^\gp_{\Ao,\et,\compln}(\Sm_k)$ similarly to \eqref{adj:RehypethatnhypethatndashvPconsthypethatnhypethatn}.
    The commutativity of the right square in \eqref{eq:diag:squares:CMobgpAetcomplncetSmkEtk} follows from the left one, and the equivalence of endofunctors $L^{\Ao,\nis}_{\et,\compln}E_{\Ao,\nis}^{\et,\compln}\simeq \mathrm{Id}_{\CMon^\gp_{\Ao,\et,\compln}(\Sm_k)}$. 
    On the other hand,
    the functor $\CMon^\gp_\compln\to\Spc$
    and the mapping $X\mapsto X(\bbC)$
    induce the commutative diagram 
    \[\xymatrix{
    \CMon^\gp_{\Ao,\Nis,\compln}(\Sm_k)\ar[r]\ar[d]&\CMon^\gp_{\Ao,\cet,\compln}(\Sm_k)\ar[r]\ar[d]&\CMon^\gp_\compln\ar[d]\\
    \Spc_{\Ao,\Nis}(\Sm_k)\ar[r]&\Spc_{\Ao,\cet}(\Sm_k)\ar[r]&\Spc
    .}\]
    Thus 
    the commutativity in the bottom two rows of \eqref{eq:RealisLoopFunctorsCommute} follows. 
%
\end{proof}

\appendix
\section{Descent spectral sequence}

Throughout the section $k$ is any field.





\begin{proposition}\label{prop:specseq_for_etalepisheaves:perfect:boundedncomplete}
    

    For any connective object $F$ in 
    $\SptSetleq(\Sm_k)$, see Notation \ref{notation:subscriptleq},
    for any 
    smooth scheme $U$ over $k$,
    there is a natural spectral sequence 
    \begin{equation}\label{eq:sepcseq:HpetUpiqFpimlFU}
    H^{p}_\cet(U,(\pi_q F)_{\cet})
    \Rightarrow
    \pi_{l} F(U) 
    \quad
    l=-p+q,
    \end{equation}
    that converges conditionally.
\end{proposition}
\begin{proof}
A modern reference for this claim 
is
\cite[Proposition 2.13]{ClausenMathewHypdescetaleKth}, a result pioneered by Thomason in \cite[Theorem 4.1]{zbMATH04175123}.
\end{proof}


\begin{lemma}\label{lm:varprojlimlLettleqlFleftarrowLetFtoLetvarprojlimltleqlF}
Let $F\in\SptS(\Sm_k)$,
then 
the natural morphisms
\begin{equation}\label{eq:varprojlimlLettleqlFleftarrowLetFtoLetvarprojlimltleqlF}
    \varprojlim_{b}\calL_{\et}t_{\leq b}F\leftarrow \calL_{\et}F\to \calL_{\et}\varprojlim_{l}t_{\leq b}F
\end{equation}
are equivalences
in $\SptS(\Sm_k)$,
where $t_{\leq b}$ is the schemewise homotopy t-structure truncation.
\end{lemma}
\begin{proof}
The right equivalence follows by $F\simeq\varprojlim_b t_{\geq b}F$.
The left morphism is a h. \'et. l. equivalence
because
\[\pi_n \calL_{\et}t_{\leq b}F(U)\simeq \begin{cases}
    0 , &n>b,\\
    \pi_n F(U), &n\leq b,
\end{cases}\]
and consequently,
$\pi_n\varprojlim_{b}\calL_{\et}t_{\leq b}F(U)\simeq \pi_n F(U)$
for any strictly henselian essentially smooth $U$.
Since
the subcategory $\SptSet(\Sm_k)$ is closed with respect to limits
by \cite[Definition 2.4 (2)]{ClausenMathewHypdescetaleKth},
and
$\calL_{\et}F$ and $\calL_{\et}t_{\leq b}F$ are \'etale hypersheaves by definition
of $\calL_\et$, see Notation \ref{notation:hypet},
the middle and left term in \eqref{eq:varprojlimlLettleqlFleftarrowLetFtoLetvarprojlimltleqlF}
are \'etale hypersheaves. 
Hence the left morphism in \eqref{eq:varprojlimlLettleqlFleftarrowLetFtoLetvarprojlimltleqlF} is an equivalence.
\end{proof}

\printbibliography
\end{document}